\documentclass[12pt,reqno]{amsart}
\usepackage{amsmath}
\usepackage{amssymb}
\usepackage{amsfonts}
\usepackage{amsthm}
\usepackage{mathrsfs}
\usepackage[all]{xy}
\usepackage{color}
\usepackage{multicol}
\usepackage{enumerate}
\usepackage{comment}

\usepackage[vcentermath,enableskew]{youngtab}
\usepackage{multirow}
\usepackage{tikz}
\usetikzlibrary{arrows}

\usepackage[colorlinks=true, pdfstartview=FitV, linkcolor=blue, citecolor=blue, urlcolor=blue]{hyperref}
\theoremstyle{plain}
\newtheorem*{theorem*}{Theorem}

\newtheorem{theorem}{Theorem}[section]
\newtheorem{lemma}[theorem]{Lemma}

\newtheorem{conj}[theorem]{Conjecture}

\theoremstyle{definition}
\newtheorem{definition}[theorem]{Definition}
\newtheorem{remark}[theorem]{Remark}

\newtheorem{example}[theorem]{Example}

\renewcommand{\Im}{{\rm Im}\,}

\newcommand{\C}{\mathbb{C}}

\newcommand{\F}{\mathbb{F}}

\newcommand{\mO}{\mathbb{O}}

\newcommand{\N}{\mathbb{N}}

\newcommand{\cB}{\mathcal{B}}
\newcommand{\cC}{\mathcal{C}}
\newcommand{\cF}{\mathcal{F}}
\newcommand{\cN}{\mathcal{N}}

\newcommand{\cX}{\mathcal{X}}
\newcommand{\cQ}{\mathcal{Q}}
\newcommand{\cS}{\mathcal{S}}
\newcommand{\cT}{\mathcal{T}}
\newcommand{\cY}{\mathcal{Y}}

\DeclareMathOperator{\End}{End}

\DeclareMathOperator{\Gl}{GL}

\DeclareMathOperator{\im}{im}
\DeclareMathOperator{\Span}{span}

\DeclareMathOperator{\Gr}{Gr}

\DeclareMathOperator{\J}{J}
\DeclareMathOperator{\eJ}{eJ}
\DeclareMathOperator{\GL}{GL}
\DeclareMathOperator{\Sp}{Sp}
\DeclareMathOperator{\Lie}{Lie}
\DeclareMathOperator{\GR}{Gr}
\DeclareMathOperator{\SYB}{SYB}
\DeclareMathOperator{\std}{std}

\newcommand{\gl}{\mathfrak{gl}}

\newcommand{\fN}{\mathfrak{N}}
\newcommand{\fZ}{\mathfrak{Z}}

\newcommand{\spp}{\mathfrak{sp}}

\newcommand{\cha}{\check{\alpha}}
\newcommand{\pperp}{\perp \!\!\! \perp}

\newcommand{\xx}{\mathrm{x}}

\newcommand{\llangle}{\langle\langle}
\newcommand{\rrangle}{\rangle\rangle}
\newcommand{\ronto}{\twoheadrightarrow}

\begin{document}
\title[Exotic Spaltenstein Varieties]{Exotic Spaltenstein Varieties}

\author{Daniele Rosso}
\address{D.~Rosso: Department of Mathematics and Actuarial Science, Indiana University Northwest}
\email{drosso@iu.edu}

\author{Neil Saunders}
\address{N.~Saunders: Department of Mathematics, School of Science and Technology, City St George's, University of London}
\email{neil.saunders.3@city.ac.uk}

\date{\today}


\maketitle

\begin{abstract} 
We define a new family of algebraic varieties, called exotic Spaltenstein varieties.
These generalise the notion of Spaltenstein varieties (which are the partial flag analogues to classical Springer fibres) to the case of exotic Springer fibres. We show that, for self-adjoint nilpotent endomorphisms of order two, the top-dimensional irreducible components are in bijection with semi-standard Young bitableaux, via constructing an explicit map. Moreover, we are able to give a combinatorial formula for this top dimension. We conjecture that this description of the irreducible components holds for nilpotent endomorphisms of arbitrary order. Finally, we mention some connections to the Robinson-Schensted-Knuth correspondence. 
\end{abstract}

\section{Introduction}

\subsection{Classical Springer Fibres and  Spaltenstein Varieties} 
Spaltenstein varieties are generalisations of the famous Springer fibres and can be defined for any algebraic group. Let $G$ be a complex reductive algebraic group and let $x$ be any element in the nilpotent cone $\mathcal{N}$  which is a singular subvariety of the Lie algebra $\Lie(G)$. Let $B$ be a Borel subgroup of $G$ and let $\mathfrak{b} = \mathfrak{h} \oplus \mathfrak{n}$ be the Lie algebra of $B$, where $\mathfrak{h}$ is the Cartan subalgebra and $\mathfrak{n}=[\mathfrak{b},\mathfrak{b}]$. There is a resolution of singularities of $\mathcal{N}$, called the Springer resolution, and over each point $x \in \mathcal{N}$ its Springer fibre can be defined as $\mathcal{F}_{x}:=\{ gB \in G/B \, | \, gxg^{-1} \in \mathfrak{n} \}$, which are subvarieties of the full flag variety $\mathcal{B} := G/B$. \vspace{5pt}

In Type A, for $G = \GL_n$, using the realisation of $\mathcal{B}$ as the variety of \emph{complete flags} of vector spaces
$$
0 =F_{0} \subset F_{1} \subset F_{2} \subset \cdots \subset F_{n} = \mathbb{C}^{n}
$$
where $\dim F_{i}=i$, we may identify the Springer fibre $\mathcal{F}_{x}$ as the subvariety of $\mathcal{B}$ stabilised by $x$, that is we have $x(F_{i}) \subset F_{i-1}$. \vspace{5pt}

The geometry of these varieties is of great interest and plays a role in modern representation theory. For example, Springer \cite{Spr78}, showed that the top degree cohomology of $\mathcal{F}_{x}$ has an action of the symmetric group of degree $n$, which makes it isomorphic to the Specht module $S^{\lambda^{tr}}$, where $\lambda$ is the Jordan type of the nilpotent element $x$ and $\lambda^{tr}$ is the transpose partition. Moreoever, De Concini and Procesi \cite{DeCPro81} showed that the total cohomology ring $H^{*}(\mathcal{F}_{x},\mathbb{C})$ is isomorphic to a quotient of the polynomial ring $\mathbb{C}[x_1, \ldots, x_n]$ by a certain ideal of elementary symmetric functions, for which there is an action of the symmetric group coming from an equivariant surjection from the cohomology of the flag variety.   \vspace{5pt}

Spaltenstein varieties arise from replacing the Borel subgroup by a parabolic subgroup $P$, and choosing a suitable Levi decomposition of $\Lie(P)$. Thus in type A, for a fixed nilpotent element $x$ and parabolic $P$, the Spaltenstein variety corresponding to the triple $(G=\GL_n,P,x)$ is defined to be:
$$
X_{P,x} = \{0 =F_{0} \subset F_{\alpha_1} \subset F_{\alpha_2} \subset \cdots \subset F_{\alpha_d} = \mathbb{C}^{n} \, | \,  \dim (F_{\alpha_{i}}/F_{\alpha_{i-1}})=\alpha_{i} \text{ and}  \, x(F_{\alpha_{i}} )\subset F_{\alpha_{i-1}}\}
$$
and where $(\alpha_{1}, \ldots, \alpha_{d})$ is a composition of $n$ (that is, a sequence of non-negative integers whose sum is $n$, see Definition \ref{def:phi}) .  \vspace{5pt}

In type A, Spaltenstein \cite{Spa76} proved that the Springer fibres and Spaltenstein varieties are pure dimensional - that is, all of the irreducible components have the same dimension. Moreover, he showed that the irreducible components are in bijection with semi-standard Young tableaux. Generalising De Concini and Procesi's results, Brundan and Ostrik \cite{BO10} give a presentation for the cohomology of the Spaltenstein varieties, which can be viewed as an analogue of the Springer fibres sitting inside the partial flag variety.  \vspace{5pt}

Outside of type A, little is known about the pure dimensionality of the $X_{P,x}$ - Spaltenstein gave an example of an $X_{P,x}$ which is not pure dimensional for an $x \in \mathfrak{so}_{8}$.  However, more recently in \cite[Theorem A]{Li20}, Li proved that for a classical group $G$ and fixed parabolic subgroup $P$ and nilpotent element $x$, the Spaltenstein varieties $X_{P,x}$ are pure dimensional if the Jordan type of $x$ is an even or odd partition - that is, of the form $(1^{m_1}3^{m_3}5^{m_5} \ldots )$ or $(2^{m_2}4^{m_4}6^{m_6} \ldots )$. By showing that the $X_{P,x}$ are Lagrangian in a certain partial resolution of a Slodowy slice, he is able to deduce further \cite[Theorem C]{Li20} that:
$$
\dim(X_{P,x}) = \tfrac{1}{2} (\dim T^{*}(G/P) - \dim \mathcal{O}_x)
$$
where $T^{*}(G/P)$ is the cotangent bundle of $G/P$ and $\mathcal{O}_{x}$ is the nilpotent orbit of $x$. 

\subsection{Kato's Exotic Nilpotent Cone and Exotic Springer Fibres}
We continue with the notation as above. The combinatorics of the Springer correspondence in Type A are particularly nice in the sense that the orbits of $\Gl_n$ on the nilpotent cone are in bijection with partitions of $n$, which parametrise the irreducible representations of its Weyl group (the symmetric group). This allowed Spaltenstein to show that there is a surjective map from $\mathcal{F}_{x}$ to $\std(\lambda)$  (standard Young tableaux of shape $\lambda$), which implies that the irreducible components of  $\mathcal{F}_{x}$ are in bijection with  $\std(\lambda)$. \vspace{5pt}

For an algebraic group outside of type A, the action on the nilpotent cone of its Lie algebra has disconnected stabilisers, which means that the Springer correspondence needs the extra data of certain local systems on the nilpotent orbits, and therefore the combinatorics becomes more complicated when trying to match with the representations of its Weyl group.  \vspace{5pt}

However, Kato's exotic nilpotent cone $\mathfrak{N}$ (see Definition \ref{def:ex-nilcone} and \cite{Kato09}) is a variation on the Type C setting which retains some of the combinatorial features seen in type A. In particular, the orbits of the action by the symplectic group $\Sp_{2n}$ on $\mathfrak{N}$ have connected stabilisers and are in bijection with the set $\mathcal{Q}_{n}$ of bipartitions of $n$, which parametrises the irreducible representations of the Weyl group of Type C (signed permutations). This enabled the authors in \cite{NRS1} to define an explicit map from the exotic Springer fibres (defined from a resolution of singularities of $\mathfrak{N}$) to standard Young bitableaux (SYB), which induced a bijection between the irreducible components of the exotic Springer fibres and the set of SYB (\cite[Theorem 2.12]{NRS1}). One consequence of this explicit map, was to recover Kato's original result that showed that the exotic Springer fibres were of pure dimension. In addition to Kato's results, the authors were able to further describe the geometric structure of these fibres in low dimension, as either being projective spaces, or projectivised line bundles over projective spaces, and also combinatorially describe when and how the closures of these irreducible components intersected (\cite[Section 8]{NRS1}.  \vspace{5pt}

Using a similar construction to Steinberg in \cite{rs}, where he demonstrated that the relative position of two Springer fibres sitting in the flag variety was determined by the classical Robinson-Schensted correspondence, the authors further exploited the map between the exotic Springer fibres and SYB to define an {\it exotic Robinson-Schensted correspondence} that is a recursive bijection between pairs of SYB of the same shape and elements of the Weyl group of type C (\cite[Section 3]{NRS2}).
\vspace{5pt}

In this paper, we define \emph{exotic Spaltenstein varieties}, which are the partial flags/parabolic subgroups version of the exotic Spring fibres of Kato, and we study their geometry. These varieties are not pure dimensional in general (see Example \ref{ex:nonpure}) but in analogy to the results of \cite{Spa76} and \cite{NRS1} we conjecture (see Conjecture \ref{conj:main}) that the top dimensional irreducible components are in bijection with semi-standard Young bitableaux, and we give a combinatorial formula for this dimension. \vspace{5pt}

The main result of the paper is that we are able to prove that this is true in the cases where the nilpotent endomorphism $x$ is such that $x^2=0$, see Theorem \ref{thm:main}. We also provide some remarks about how a proof might proceed in general (see Section \ref{section:extend}). \vspace{5pt}

Finally, our bijection between top-dimensional components of exotic Spaltenstein varieties and semi-standard Young bitableaux would give rise to an {\it exotic Robinson-Schensted-Knuth correspondence}, which should be a generalisation of our exotic Robinson-Schensted algorithm from \cite{NRS2}, in a way that is analogous to the Type A case studied in \cite{Ros12} (see Section \ref{section:ersk}).

\subsection*{Acknowledgments}D.R. was supported in this research by a Summer Faculty Fellowship and a Grant-in-Aid of research from Indiana University Northwest. N.S. was supported by the London Mathematical Society’s Scheme 4 \emph{Research in Pairs} grant, ref. No. 41939. 

\section{Background and Notation}\label{sec:notation}

\subsection{Symplectic Spaces and Isotropic Grassmannians}
We let $\N=\{0,1,2,\ldots,\}$ be the set of nonnegative integers. All vector spaces will be assumed to be over the field $\C$ of complex numbers.
Throughout the paper, we will have $n\in \N$ and $V$ will be a $2n$-dimensional symplectic vector space, i.e. it is equipped with a nondegenerate, skew-symmetric bilinear form $\langle ~,~\rangle: V\times V\to \C$. If $W\subset V$, we denote by $W^\perp$ its perpendicular space with respect to the form $\langle ~,~\rangle$.

We define the symplectic group as the group of invertible linear transformations of $V$ preserving the form
$$\Sp_{2n}=\Sp(V):=\{g\in \Gl(V)~|~\langle gv, gw \rangle = \langle v,w\rangle, ~\forall v,w\in V\}.$$

\begin{definition}\label{def:Grass} For $0\leq k\leq n$, we define the \emph{isotropic Grassmannian variety} of all $k$-dimensional isotropic subspaces of $V$ by
$$\GR^\perp_{k,2n}= \GR^\perp_k(V)=\{ F\subset V~|~\dim(F)=k,~F\subset F^\perp \}.$$
\end{definition}
\begin{remark}\label{rem:grass-dim}Notice that $\Sp_{2n}$ acts transitively on $\GR^\perp_{k,2n}$, and it follows from \cite[Prop 4.4]{BKT} that
$\GR^\perp_{k,2n}$ is an irreducible projective variety of dimension $k(2n-k)-\frac{k(k-1)}{2}$.     
\end{remark}
If we restrict the bilinear form to a subspace $W\subset V$, in general it might be degenerate. So we will also need to consider degenerate skew-symmetric forms.
\begin{definition}Let $W$ be a vector space of dimension $d$, with a skew-symmetric bilinear form $\langle~,~\rangle$, possibly degenerate. Consider the radical $\operatorname{rad}\langle~,~\rangle=\{w\in W~|~\langle w,u \rangle=0,~\forall u\in W\}$. Then $W/\operatorname{rad}\langle~,~\rangle$ is a symplectic space with the induced form on the quotient hence even dimensional. We say that $2r=\dim\left(W/\operatorname{rad}\langle~,~\rangle\right)$ is the \emph{rank} of the bilinear form, clearly $0\leq r\leq \frac{d}{2}$. 

In this case, for $0\leq k\leq d-r$, we can define the (degenerate) isotropic Grassmannian
$$\GR^{\perp r}_{k,d}= \GR^\perp_k(W)=\{ F\subset W~|~\dim(F)=k,~F\subset F^\perp \}.$$
Note that for $d=2n$ and $r=n$, we recover the previous Definition \ref{def:Grass}.
\end{definition}
\begin{lemma}\label{lem:dim-deg-grass}The degenerate isotropic Grassmannian $\GR^{\perp r}_{k,d}$, for all $0\leq r\leq \frac{d}{2}$ and $0\leq k\leq d-r$ is an irreducible projective variety of dimension $k(d-k)-\frac{k(k-1)}{2}+\frac{(k-r)(k-r-1)}{2}$, with the last term being zero when $k\leq r$.
\end{lemma}
\begin{proof}
We use the same idea as \cite[\S 4.8]{Mi}. Let $\{e_1,\ldots,e_d\}$ be a basis of $W$ with the property that $e_1,\ldots,e_{d-2r}$ is a basis of $\operatorname{rad}\langle~,~\rangle$, and such that $\langle e_i, e_j\rangle= \delta_{i+j,2d-2r+1}$ for all $1\leq i\leq d-r$. We embed $W$ in a symplectic vector space $\tilde{W}$, of dimension $2d-2r$, with basis $\{e_1,\ldots,e_{2d-2r}\}$ and non degenerate skew-symmetric bilinear form $\langle~,~\rangle_{\sim}$ defined by $\langle e_i, e_j\rangle_\sim= \delta_{i+j,2d-2r+1}$ for all $1\leq i\leq d-r$. In this way, the degenerate bilinear form on $W$ becomes the restriction of the symplectic form $\langle~,~\rangle_{\sim}$. Then the degenerate isotropic Grassmannian $\GR^\perp_k(W)$ can be identified with a subvariety of the isotropic Grassmannian in $\tilde{W}$, more precisely
\begin{equation}\label{eq:schubert}
\GR^\perp_k(W)=\{F\in \GR^\perp_k(\tilde{W})~|~F\subset W\}.
\end{equation}
As in \cite[\S 4.8]{Mi}, this is a Schubert variety in $\GR^\perp_k(\tilde{W})$, i.e. the closure of a Schubert cell. We can distinguish two cases. When $k\leq r$, then \eqref{eq:schubert} is the Schubert variety $X_{i_1,\ldots,i_k}$ for the index set $i_1=d-k+1,\ldots, i_k=d$. When $r<k\leq d-r$, it's the Schubert variety for the index set $i_1=d-r-k+1,\ldots,i_{k-r}=d-2r$, $i_{k-r+1}=d-r+1,\ldots,i_k=d$. In both cases, since the Schubert cell is isomorphic to an affine space, it is irreducible and so it its closure. 

To conclude the proof we just need to compute the dimensions of the Schubert cells in the two cases. To do this, we will instead use the formulas from the proof of \cite[Prop 4.4]{BKT} (it would also be possible to use \cite[Prop 4.4]{BKT} directly by first translating the index set into a partition). Each index $i_j$ gives a dimension of $i_j-j-\#\{\ell<j~|~i_\ell+i_j>2d-2r+1\}$, hence
\begin{align*}\dim X_{i_1,\ldots,i_k}&=\sum_{j=1}^k(i_j-j-\#\{\ell<j~|~i_\ell+i_j>2d-2r+1\})\\
&=\left(\sum_{j=1}^ki_j\right)-\frac{k(k+1)}{2}-\#\{(\ell,j)~|~1\leq \ell<j\leq k,~ i_\ell+i_j>2d-2r+1\}.
\end{align*}
When $k\leq r$, then $i_j=d-k+j$, so $\#\{(\ell,j)~|~1\leq \ell<j\leq k,~ i_\ell+i_j>2d-2r+1\}={k \choose 2}=\frac{k(k-1)}{2}$, and we get indeed that
$$\dim X_{i_1,\ldots,i_k}=k(d-k)-\frac{k(k-1)}{2}.$$
When $k>r$, then $i_j=d-r-k+j$ for $1\leq j\leq k-r$, and $i_j=d-k+j$ for $k-r+1\leq j\leq k$, so  $\#\{(\ell,j)~|~1\leq \ell<j\leq k,~ i_\ell+i_j>2d-2r+1\}={r \choose 2}=\frac{r(r-1)}{2}$ and with a short computation we get indeed that 
$$\dim X_{i_1,\ldots,i_k}=k(d-k)-\frac{k(k-1)}{2}+\frac{(k-r)(k-r-1)}{2}.$$
\end{proof}
\begin{remark}\label{rem:third-term}
Notice that if $k=r+1$, the term $\frac{(k-r)(k-r-1)}{2}$ is again equal to zero, so in the formula of Lemma \ref{lem:dim-deg-grass}, the third term can be ignored whenever $k\leq r+1$.
\end{remark}

\subsection{Exotic Nilpotent Cone}
Let $V$ be a $2n$-dimensional symplectic space $V$, with symplectic group $\Sp_{2n}=\Sp(V)$, we identify its Lie algebra as follows
$$\spp_{2n}:=\Lie(\Sp_{2n})=\{x\in \End(V)~|~\langle xv,w\rangle + \langle v,xw\rangle =0, ~\forall v,w\in V\}.$$

The adjoint action of $\Sp_{2n}$ on $\spp_{2n}$ can be interpreted as the restriction of the $\Sp_{2n}$-action on $\gl_{2n}=\End(V)$ given by conjugation. This action gives a direct sum decomposition of $\Sp_{2n}$-modules
$\gl_{2n}=\spp_{2n}\oplus \cS$, where $\cS$ can also be described directly as
$$\cS=\{x\in \End(V)~|~\langle xv,w\rangle =\langle v,xw\rangle, ~\forall v,w\in V\}.$$
Finally we let $\cN:=\{x\in\End(V)~|~x^k=0,\text{ for some }k\}$ be the \emph{nilpotent cone} of the Lie algebra of $\Gl_{2n}=\Gl(V)$.

\begin{definition}\label{def:ex-nilcone}
The \emph{exotic nilpotent cone} is the (singular) variety $\mathfrak{N}:=V \times (\mathcal{S} \cap \mathcal{N})$.
\end{definition}

When it is not clear from the context, we might specify the underlying vector space and write $\cS(V)$, $\cN(V)$ or $\fN(V)$.

If $k\in \N$, a \emph{composition} of $k$ is a finite sequence $\alpha=(\alpha_1,\alpha_2,\ldots,\alpha_m)$ such that for all $1\leq i\leq m$, $\alpha_i$ is a positive integer, and $\sum_{i=1}^m\alpha_i=k$. We denote this by $|\alpha|=k$ or $\alpha\vDash k$. If $\alpha=(\alpha_1,\alpha_2,\ldots,\alpha_m)$, then $\ell(\alpha):=m$, the \emph{length} of the composition is the number of its parts. 

A \emph{partition} of $k$ is a composition of $k$ where the integers are weakly decreasing, that is $\lambda=(\lambda_1,\ldots,\lambda_m)$ is such that $\lambda_1\geq\ldots\geq\lambda_m>0$ and $\sum_{i=1}^k\lambda_i=k$. We denote it by $\lambda \vdash k$ or $|\lambda|=k$. On occasion it will be convenient for us to consider a partition to have some (or even infinitely many) parts of size zero appended to its end, this does not change the length of the partition which is the number of nonzero parts.

\begin{definition}\label{def:jordan}If $W$ is a vector space and $x\in \End(W)$ is a nilpotent transformation, its Jordan canonical form gives a partition of $\dim(W)$ by considering the sizes of the Jordan blocks (in weakly decreasing order). We denote this partition by $\J(x)$ and we call it the \emph{Jordan type} of $x$.
\end{definition}

If we have two partitions $\mu$, $\nu$ (of potentially different integers), we define two new partitions $\mu+\nu:=(\mu_1+\nu_1,\mu_2+\nu_2,\ldots)$, and $\mu\cup\nu$ which is the unique partition obtained by reordering the sequence $(\mu_1,\ldots,\mu_{\ell(\mu)},\nu_1,\ldots,\nu_{\ell(\nu)})$ to make it weakly decreasing. A \emph{bipartition} of $n$ is a pair $(\mu,\nu)$ of partitions such that $|\mu|+|\nu|=n$. We denote the set of all bipartitions of $n$ by $\cQ_n$.

The set $\cQ_n$ is important for us because of the following result.

\begin{theorem}[{\cite[Thm 6.1]{AH}}] The orbits of the symplectic group $\Sp_{2n}$ on the exotic nilpotent cone $\fN$ are in bijection with $\cQ_n$.
\end{theorem}

More precisely, following Section 2 and Section 6 of \cite{AH} we can say that, given a bipartition $(\mu,\nu)\in\cQ_n$, the corresponding orbit $\mO_{(\mu,\nu)}$ contains the point $(v,x)$ if and only if there is a `normal' basis of $V$ given by $$\{v_{ij}, v_{ij}^{*} \, | \, 1 \leq i \leq \ell(\mu+\nu), 1 \leq j \leq \mu_i+\nu_i \},$$ with $\langle v_{ij},v^*_{i'j'}\rangle=\delta_{i,i'}\delta_{j,j'}$, $v=\sum_{i=1}^{\ell(\mu)}v_{i,\mu_i}$
and such that the action of $x$ on this basis is as follows: 
$$xv_{ij}=\begin{cases} v_{i,j-1}  & \mbox{if} \quad j \geq 2 \\ 0  &\mbox{if} \quad j=1 \end{cases} \quad \quad \quad xv_{ij}^{*}=\begin{cases} v_{i,j+1}^{*}  & \mbox{if} \quad j \leq \mu_i+\nu_i-1 \\ 0  &\mbox{if} \quad  j=\mu_i+\nu_i \end{cases},$$ in particular the Jordan type of $x$ is $(\mu+\nu)\cup(\mu+\nu)$.

\begin{definition}\label{def:exjordan}
Similarly to Definition \ref{def:jordan}, if $(v,x)\in\mO_{(\mu,\nu)}$, as defined above, we say that the bipartition $(\mu,\nu)$ is the \emph{exotic Jordan type} of $(v,x)$ and we denote it by $\eJ(v,x)=(\mu,\nu)$.
\end{definition}

Associated to a partition $\lambda=(\lambda_i)_i$ there is a Young diagram consisting of $\lambda_i$ boxes on row $i$. We say that $\lambda$ is the \emph{shape} of the diagram. In the same way, a bipartition gives a pair of diagrams. If a Young diagram is filled with positive integers, we call it a Young tableau, and similarly, if we fill a pair of diagrams with positive integers we call it a bitableau.

\begin{example}
The bipartition $(\mu,\nu)=((3,1),(2,2,1))$ corresponds to the pair of Young diagrams $\left(\young(~~~,::~),\yng(2,2,1)\right)$, notice that we follow the usual convention of left-justifying the boxes in the second diagram, however we instead right-justify the boxes in the first diagram (equivalenty, we apply a vertical reflexion to a usual diagram). The reason for this comes from the look of normal bases, as in Example \ref{example:boxes}.
An example of a bitableau in that shape would be $\left(\young(331,::2),\young(11,12,3)\right)$.
\end{example}
\begin{example} \label{example:boxes}
Consider the bipartition $(\mu,\nu)=((3,1),(2,2,1))$, and a point $(v,x)\in\mO_{(\mu,\nu)}$. Then we can represent the `normal' basis as

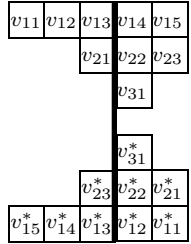
\begin{center}
\begin{tikzpicture}
\draw(0,0) node {$\young(\hfil\hfil\hfil,::\hfil)$};
\draw[-,line width=2pt] (0.7,0.47) to (0.7,-2.72);
\draw(1.2,-0.23) node {$\yng(2,2,1)$};	
\draw(1.2,-2) node {$\yng(1,2,2)$};
\draw(0,-2.24) node {$\young(::\hfil,\hfil\hfil\hfil)$};
\draw(-0.47,0.2) node {{\tiny $v_{11}$}};
\draw(0,0.2) node {{\tiny $v_{12}$}};
\draw(0.47,0.2) node {{\tiny $v_{13}$}};
\draw(0.94,0.2) node {{\tiny $v_{14}$}};
\draw(1.41,0.2) node {{\tiny $v_{15}$}};
\draw(0.47,-0.27) node {{\tiny $v_{21}$}};
\draw(0.94,-0.27) node {{\tiny $v_{22}$}};
\draw(1.41,-0.27) node {{\tiny $v_{23}$}};
\draw(0.94,-0.74) node {{\tiny $v_{31}$}};
\draw(0.94,-1.54) node {{\tiny $v^*_{31}$}};
\draw(0.47,-2.01) node {{\tiny $v^*_{23}$}};
\draw(0.94,-2.01) node {{\tiny $v^*_{22}$}};
\draw(1.41,-2.01) node {{\tiny $v^*_{21}$}};
\draw(-0.47,-2.48) node {{\tiny $v^*_{15}$}};
\draw(0,-2.48) node {{\tiny $v^*_{14}$}};
\draw(0.47,-2.48) node {{\tiny $v^*_{13}$}};
\draw(0.94,-2.48) node {{\tiny $v^*_{12}$}};
\draw(1.41,-2.48) node {{\tiny $v^*_{11}$}};
\end{tikzpicture}
\end{center}

with the action of $x$ given by moving one block left (and zero if there is nothing further left), and $v=v_{13}+v_{21}$ is given by the sum of the boxes just left of the dividing wall on the upper half of the diagram. Notice that the wall divides the two diagrams corresponding to the partitions that form the bipartition (the one on the left of the wall is facing backwards) and that the bipartition is repeated twice since $\J(x)=(\mu+\nu)\cup(\mu+\nu)$.
\end{example}

\subsection{Exotic Spaltenstein varieties}
We are now going to define the main object of study of this article and state our main conjecture about the irreducible components of exotic Spaltenstein varieties, which we will be able to prove in the case where $x^2=0$.

Let $\alpha=(\alpha_1,\alpha_2,\ldots,\alpha_m)\vDash n$ be a composition, we define a composition of $2n$ with $2m$ parts, $\hat{\alpha}=(\alpha_1,\alpha_2,\ldots, \alpha_m,\alpha_m,\alpha_{m-1},\ldots, \alpha_1)\vDash 2n$. We also define, $\check{\alpha}_i=\sum_{j=1}^i\hat{\alpha}_j$ for $i=0,\ldots,2m$. We denote by $\cF^\alpha(V)$ the variety of \emph{partial symplectic flags} of type $\alpha$ in $V$, that is each $F_\bullet\in\cF^\alpha(V)$ is a sequence of subspaces
$$ F_\bullet=(0=F_0\subseteq F_{\check{\alpha}_1}\subseteq \cdots\subseteq F_{\cha_m} \subseteq \cdots \subseteq F_{\cha_{2m-1}}\subseteq F_{\cha_{2m}}=V)$$
such that $\dim(F_{\cha_i}/F_{\cha_{i-1}})=\hat{\alpha}_i$ and $F_{\cha_i}^\perp=F_{\cha_{2m-i}}$ for all $1\leq i\leq 2m$. 

Notice that $F_{\cha_m}=F_n=F_n^{\perp}$ is a maximal isotropic subspace of $V$.

\begin{definition}
Given $(v,x)\in\fN$, and $\alpha=(\alpha_1,\ldots,\alpha_m)\vDash n$, we define the \emph{exotic Spaltenstein variety}
$$\cC^\alpha_{(v,x)}:=\{F_\bullet=(0\subseteq F_{\check{\alpha}_1}\subseteq\cdots\subseteq F_{\check{\alpha}_{2m}}=V)\in \cF^\alpha(V)~|~v\in F_n,~x(F_{\check{\alpha}_i})\subseteq F_{\check{\alpha}_{i-1}}\}.$$
\end{definition}
\begin{remark}
    Notice that for certain choices of $v$, $x$ and $\alpha$ the variety could be empty. It is clear that if $\eJ(v,x)=\eJ(v',x')$, then there is an isomorphism of varieties $\cC^\alpha_{(v,x)}\cong\cC^\alpha_{(v',x')}$ given by the $\Sp_{2n}$-action. When $\alpha=1^n$, then the partial flags are actually complete flags and this variety is an exotic Springer fiber, as defined by Kato in \cite{Kat} and studied by the authors in \cite{NRS1,NRS2}.
\end{remark}
Let $(v,x)\in\fN$, if $F_\bullet\in\cC^\alpha_{(v,x)}$, for all $i=0,\ldots,m$, since $F_{\cha_i}$ and $F_{\cha_i}^\perp=F_{\cha_{2m-i}}$ are both invariant under $x$, we can consider the restriction of $x$ to $F_{\cha_i}^\perp / F_{\cha_i}$, which is a symplectic vector space of dimension $2(n-\cha_i)$ (since $F_{\cha_i}$ is isotropic, the bilinear form descends to the quotient and is nondegenerate because we restrict to $F_{\cha_i}^\perp$). It can be easily verified that $x|_{F_{\cha_i}^\perp / F_{\cha_i}}\in\cN( F_{\cha_i}^\perp / F_{\cha_i})\cap\cS (F_{\cha_i}^\perp/F_{\cha_i})$.

\begin{definition}\label{def:phi}Let $\alpha=(\alpha_1,\alpha_2,\ldots,\alpha_m)\vDash n$ be a composition. Let $(v,x)\in\fN$, $\eJ(v,x)=(\mu,\nu)$.
Define the following maps,  
\begin{align*} \Phi^i:\cC^\alpha_{(v,x)}\to \cQ_{n-\cha_{i}} &\qquad F_\bullet \mapsto \eJ\left(v+F_{\cha_i}, x|_{F_{\cha_i}^\perp/F_{\cha_i}}\right)\qquad \text{$\forall$ $0\leq i\leq m$}; \\ \Phi:\cC^\alpha_{(v,x)}\to \prod_{i=m}^0\cQ_{n-\cha_i} &\qquad F_\bullet\mapsto \left((\emptyset,\emptyset)=\Phi^m(F_\bullet),\Phi^{m-1}(F_\bullet),\ldots,\Phi^{1}(F_\bullet),\Phi^0(F_\bullet)=(\mu,\nu)\right). \end{align*} \end{definition} 
This map assigns to each partial flag in the exotic Spaltenstein variety a sequence of bipartitions, such that the total size increases by $\alpha_i$ boxes at each step. Notice that in general there is no guarantee that the shape of $\Phi^{i}(F_\bullet)$ will contain the shape of $\Phi^{i+1}(F_\bullet)$. If the sequence of bipartitions is nested at each step, we can identify the sequence with a bitableau of shape $(\mu,\nu)$, where at each step $i=m-1,m-2,\ldots,0$ we fill the boxes that have been added with the number $m-i$.

\begin{example}\label{exa:tableaux}  
\begin{enumerate}
\item Let $n=2$, $(\mu,\nu)=\left(1^2,\emptyset\right)=\left(\yng(1,1),\emptyset\right)$, and $(v,x)\in\mO_{(\mu,\nu)}$, then we have a basis $\{v_{11},v_{21},v_{21}^*,v_{11}^*\}$ for $V$, with $v=v_{11}+v_{21}$. If $\alpha=(2)$, and $F_\bullet=(0\subseteq F_2\subseteq V)$ with $F_2=\C v_{11}+\C v_{21}$ (or any subspace of $V$ containing $v$), then $F_\bullet\in \cC^\alpha_{(v,x)}$, and since $F_2^\perp/F_2=0$ we have $\Phi(F_\bullet)=\left((\emptyset,\emptyset),\left(\yng(1,1),\emptyset\right)\right).$ This sequence is nested and corresponds to the bitableau $\left(\young(1,1),\emptyset\right)$.

\item With $(\mu,\nu)$ as above, if $\alpha=(1,1)$ and $F_\bullet=(0\subseteq F_1\subseteq F_2\subseteq F_1^\perp\subseteq V)$ with $F_1=\C v$, $F_2=\C v_{11}+\C v_{21}$, then since $v\in F_1$ we have that $v+F_1$ is zero in the quotient $F_1^\perp/F_1$, which is a $2$ dimensional space, while $F_2^\perp/F_2=0$, so
$\Phi(F_\bullet)=\left((\emptyset,\emptyset),(\emptyset,\yng(1)),\left(\yng(1,1),\emptyset\right)\right).$ This sequence of bipartitions is not nested, hence it does not correspond to a bitableau.

\item Now let $(\mu,\nu)=\left(\emptyset,2\right)=\left(\emptyset,\yng(2)\right)$, and $(v,x)\in\mO_{(\mu,\nu)}$, then we have a basis $\{v_{11},v_{12},v_{12}^*,v_{11}^*\}$ for $V$, and $v=0$. If $\alpha=(2)$, and $F_\bullet=(0\subseteq F_2\subseteq V)$ with $F_2=\C v_{11}+\C v_{12}^*$, then $F_\bullet\in \cC^\alpha_{(v,x)}$, and since $F_2^\perp/F_2=0$ we have $\Phi(F_\bullet)=\left((\emptyset,\emptyset),\left(\emptyset,\yng(2)\right)\right)$, which is a nested sequence corresponding to the bitableau $\left(\emptyset,\young(11)\right)$.
\end{enumerate}
\end{example}
\begin{definition}
Let $\alpha=(\alpha_1,\ldots,\alpha_m)\vDash n$, we say that a bitableau of shape $(\mu,\nu)\in \cQ_n$ is \emph{semistandard} with content $\alpha$ if for all $i=1,\ldots,m$ it contains $\alpha_i$ boxes with the number $i$, such that the numbers are strictly increasing along rows (right to left in the first tableau, left to right in the second tableau) and weakly increasing down columns. Equivalently, if we consider the bitableau as a nested sequence of bipartitions, it is semistandard if at each step we add a vertical strip to each of the two Young diagrams (no two boxes added are in the same row of the same diagram). 

We denote by $\cT^\alpha_{\mu,\nu}$ be the set of all semistandard bitableaux of shape $(\mu,\nu)$ and content $\alpha$.
\end{definition}
\begin{remark}
Notice that our convention is the transpose of the usual convention for semistandard tableaux (usually they are taken to be increasing strictly down columns and weakly along rows) but it matches the convention used in \cite{Ros12}. In Example \ref{exa:tableaux}, the bitableau in case (1) is semistandard, while the one in case (3) is not. 
\end{remark}
Let $\lambda=(\lambda_i)_i\vdash n$ be a partition, we define $N(\lambda)=\sum_{i}(i-1)\lambda_i$. If $(\mu,\nu)\in\cQ_n$ is a bipartition and $\alpha=(\alpha_1,\ldots,\alpha_m)\vDash n$ is a composition, we define
\begin{equation}\label{eq:d-alpha}d^{\alpha}_{(\mu,\nu)}:=2N(\mu+\nu)+|\nu|-\frac{1}{2}\sum_{i=1}^m(\alpha_i^2-\alpha_i).\end{equation}
\begin{remark}Recall that $N(\lambda)$ is the dimension of the Springer fibre over a nilpotent element of Jordan type $\lambda$, while when $\alpha=1^n$, we have $d^{(1^n)}_{\mu,\nu}=2N(\mu+\nu)+|\nu|$ is the dimension of the exotic Springer fibre over a pair $(v,x)\in\cN$ with $\eJ(v,x)=(\mu,\nu).$    
\end{remark}
\begin{conj}\label{conj:main}Let $T\in\cT^\alpha_{\mu,\nu}$ be a semistandard bitableau of shape $(\mu,\nu)$ and content $\alpha$, and let $(v,x)\in\mO_{(\mu,\nu)}$. Then $\overline{\Phi^{-1}(T)}$ is an irreducible component of $\cC^\alpha_{(v,x)}$ of dimension $d^\alpha_{\mu,\nu}$ and all irreducible components of dimension $d^\alpha_{\mu,\nu}$ are of this form. All other irreducible components of $\cC^\alpha_{(v,x)}$ have dimension strictly lower than $d^\alpha_{\mu,\nu}$.
\end{conj}
\begin{remark}
In the case of complete flags, when $\alpha=1^n$, the variety is an exotic Springer fibre and the result is true (see \cite{NRS1}). In that case all irreducible components have the same dimension, but in general $\cC^\alpha_{(v,x)}$ is not pure dimensional, as the following example shows.
\end{remark}
\begin{example}\label{ex:nonpure}
Let $n=3$, $(\mu,\nu)=(\emptyset,21)$, $\alpha=(1,2)$, then $v=0$ and we have $F_\bullet\in \cC^\alpha_{(v,x)}$ if $F_\bullet=(0\subset F_1\subset F_3=F_3^\perp\subset F_1^\perp\subset V)$ with $F_1\in\ker(x)$ and $x(F_3)\subset F_1$. The subvariety of $\cC^\alpha_{(v,x)}$ where $F_1\subset\Im(x)$ is an irreducible component of dimension $d^\alpha_{(\mu,\nu)}=2(1)+3-\frac{1}{2}(1^2-1)-\frac{1}{2}(2^2-2)=2+3-0-1=4$, in fact it equals $\Phi^{-1}\left(\emptyset,\young(12,1)\right)$. The subvariety of $\cC^\alpha_{(v,x)}$ where $F_1\not\subset\Im(x)$ however is an irreducible locally closed subvariety, whose closure is an irreducible component of dimension $3$, for these flags the corresponding sequence of bipartitions is still nested, 
$$\Phi(F_\bullet)=\left((\emptyset,\emptyset),\left(\emptyset,\yng(2)\right),\left(\emptyset,\yng(2,1)\right)\right)=\left(\emptyset,\young(11,2)\right)$$
but the resulting bitableau is not semistandard.
\end{example}

\subsection{Strategy for the proof of Conjecture \ref{conj:main}}
We induct on $m$, the number of parts of the composition $\alpha=(\alpha_1,\alpha_2,\ldots,\alpha_m)\vDash n$. For $m=0$, then $n=0$ as well and the statement is trivial (all the varieties involved are one single point), so we can assume that $m\geq 1$.

Let $(\mu,\nu)\in\cQ_n$, $\alpha=(\alpha_1,\alpha_2,\ldots,\alpha_m)\vDash n$, $(v,x)\in\mO_{(\mu,\nu)}$, and let $B=((\mu^{[m]},\nu^{[m]}),\ldots,(\mu^{[0]},\nu^{[0]}))$ be a sequence of bipartitions such that $(\mu^{[i]},\nu^{[i]})\in\cQ_{n-\cha_i}$, for $i=m,\ldots,0$ and $(\mu^{[0]},\nu^{[0]})=(\mu,\nu)$. Then if $\emptyset\neq \Phi^{-1}(B)\subseteq \cC^\alpha_{v,x}$ we have a map
\begin{eqnarray*}
p:\Phi^{-1}(B)	&\longrightarrow& \Gr_{\alpha_1}^\perp (\ker(x) \cap (\C[x]v)^{\perp}); \\
	F_\bullet 	&\mapsto& F_{\alpha_1}.
\end{eqnarray*}
Define $(\mu',\nu')=(\mu^{[1]},\nu^{[1]})=\Phi^{1}(F_\bullet)=\eJ\left(v+F_{\alpha_1},x|_{F_{\alpha_1}^\perp/F_{\alpha_1}}\right)$, $\alpha'=(\alpha_2,\ldots,\alpha_m)$, and $B'=((\mu^{[m]},\nu^{[m]}),\ldots,(\mu^{[1]},\nu^{[1]}))$. We let
$$ \cB_{((\mu',\nu'),\alpha')}^{((\mu,\nu),\alpha)}:=\im p=\{F \in \Gr_{\alpha_1}^\perp (\ker(x) \cap (\C[x]v)^{\perp})\, | \,\eJ\left(v+F,x|_{F^\perp/F}\right)=(\mu',\nu')\, \}. $$
For any $F\in\cB_{((\mu',\nu'),\alpha')}^{((\mu,\nu),\alpha)}$, we then have that 
\begin{align}
\notag p^{-1}(F)&=\{F'_\bullet\in \Phi^{-1}(B)~|~F'_{\alpha_1}=F\} \\
\label{eq:p-proj} &\cong \left\{\overline{F}_\bullet\in\cC^{\alpha'}_{\left(v+F,x|_{F^\perp/F}\right)}\subseteq\cF^{\alpha'}(F^\perp/F)~|~(\Phi^{m-1}(\overline{F}_\bullet),\ldots,\Phi^0(\overline{F}_\bullet))=B'\right\} 
\end{align}
where the isomorphism  is given by 
$$(F'_0,F'_1,\ldots,F'_{2m-1},F'_{2m})\mapsto (\overline{F}_0,\overline{F}_1,\ldots,\overline{F}_{2m-3},\overline{F}_{2m-2});\quad \overline{F}_i=F'_{i+1}/F\quad\forall i=0,\ldots,2m-2.$$
In conclusion, the map $p$ is a fibre bundle, and from \eqref{eq:p-proj}, it follows that the fibres are $p^{-1}(F)\cong \Phi^{-1}(B')$, which by inductive hypothesis is an irreducible variety of dimension lesser or equal to $d^{\alpha'}_{\mu',\nu'}$, with equality if and only if the sequence $B'$ is actually a semistandard tableau. Thus Conjecture \ref{conj:main} would be a consequence of the following result.
\begin{conj}\label{prop:B-irred} For all $(\mu,\nu)\in\cQ_n$ and for all $(\mu',\nu')\in\cQ_{n-\alpha_1}$ 
the variety $ \cB_{((\mu',\nu'),\alpha')}^{((\mu,\nu),\alpha)}$ is irreducible, and we have 
$$\dim\cB_{((\mu',\nu'),\alpha')}^{((\mu,\nu),\alpha)}\leq d^\alpha_{(\mu,\nu)}-d^{\alpha'}_{(\mu',\nu')},$$
with equality if and only if $(\mu',\nu')$ is obtained from $(\mu,\nu)$ by removing $\alpha_1$ boxes in a vertical strip in each diagram (no two of the boxes are in the same row of the same diagram).
\end{conj}
In Section \ref{sec:proof}, we will prove that Conjecture \ref{prop:B-irred} holds when $x^2=0$, hence our main result for this paper is the following.
\begin{theorem} \label{thm:main} Let $\alpha\vDash n$, $(\mu,\nu)\in \cQ_n$, and $(v,x)\in\mO_{(\mu,\nu)}$, with $x^2=0$. For all $T\in\cT^\alpha_{\mu,\nu}$, $\overline{\Phi^{-1}(T)}$ is an irreducible component of $\cC^\alpha_{(v,x)}$ of dimension $d^\alpha_{\mu,\nu}$ and all irreducible components of dimension $d^\alpha_{\mu,\nu}$ are of this form. All other irreducible components of $\cC^\alpha_{(v,x)}$ have dimension strictly lower than $d^\alpha_{\mu,\nu}$.
\end{theorem}

\section{Computing Jordan types}\label{sec:jordan}
It will be important for us to compute the Jordan type of a linear transformation when restricting to (or quotienting by) invariant subspaces. Recall that if $V$ is any vector space (not necessarily symplectic) and $x\in \End(V)$ is a nilpotent linear transformation, then $\J(x)=\lambda$, where for all $j\geq 1$, $\sum_{k=1}^j\lambda_j^t=\dim(\ker(x^j))$. Here $\lambda^t$ denotes the transposed partition, with $\lambda_j^t=\#\{k~|~\lambda_k\geq j\}$. Equivalently, we have that the number of boxes in the $j$-th column of the Young diagram of $\J(x)$ is $\lambda_j^t=\dim(\ker(x^j))-\dim(\ker(x^{j-1}))$ for all $j\geq 1$.
\begin{lemma}\label{lem:general-res}
Let $0\subseteq U\subseteq W\subseteq V$ be subspaces, and let $x\in\End(V)$ be a nilpotent linear transformation with $x(U)\subseteq U$, $x(W)\subseteq W$. Then we can consider the restriction to the subquotient, and $\J(x|_{W/U})=\mu$ where
$$\mu_i^t=\dim(x^i(W)\cap U)+\dim(\ker(x^i)\cap W)-\dim(x^{i-1}(W)\cap U)-\dim(\ker(x^{i-1})\cap W)$$
\end{lemma}
\begin{proof}
We know that 
$\mu_i^t=\dim\left(\ker(x|_{W/U})^i\right)-\dim\left(\ker(x|_{W/U})^{i-1}\right)$.
Since $x^i(U)\subseteq U$, we have that $(x^i)^{-1}(U)\supseteq U$. It follows that  
$$ \ker(x|_{W/U})^i\simeq \left( (x^i)^{-1}(U)\cap W\right)/U$$
so
$$ \dim(\ker(x|_{W/U})^i)=\dim \left( (x^i)^{-1}(U)\cap W\right)-\dim (U).$$
But, considering the map $x^i:(x^i)^{-1}(U)\cap W\to U$ we get an isomorphism
$$ \left((x^i)^{-1}(U)\cap W\right)/(\ker(x^i)\cap W)\simeq x^i(W)\cap U$$
hence
$$ \dim \left((x^i)^{-1}(U)\cap W\right)=\dim(x^i(W)\cap U)+\dim(\ker(x^i)\cap W).$$
In conclusion, we have
\begin{align*}
\mu_i^t&=\dim\left(\ker(x|_{W/U})^i\right)-\dim\left(\ker(x|_{W/U})^{i-1}\right)\\
&= \left(\dim \left( (x^i)^{-1}(U)\cap W\right)-\dim (U)\right)-\left(\dim \left( (x^{i-1})^{-1}(U)\cap W\right)-\dim (U)\right)\\
&= \dim \left( (x^i)^{-1}(U)\cap W\right)-\dim \left( (x^{i-1})^{-1}(U)\cap W\right)\\
&= \dim(x^i(W)\cap U)+\dim(\ker(x^i)\cap W)-\dim(x^{i-1}(W)\cap U)-\dim(\ker(x^{i-1})\cap W).
\end{align*}
\end{proof}
\begin{lemma}\label{lem:jtype-quot}Let $x$ be a nilpotent transformation on a vector space $V$, and let $F\subseteq\ker(x)$. Define 
\begin{equation}\label{eq:adim}a_j=\dim (F\cap \Im(x^j))\quad j\geq 0.\end{equation}
Then $\J(x|_{V/F})$ is obtained from $\J(x)$ by removing $(a_{j-1}-a_j)$ boxes at the bottom of column $j$ for all $j\geq 1$. (Notice that these boxes form a vertical strip i.e. no two of them are in the same row)
\end{lemma}
\begin{proof}This follows directly from applying Lemma \ref{lem:general-res} to $U=F$, $W=V$ and comparing the result to $\dim(\ker(x^j))-\dim(\ker(x^{j-1}))$.
\end{proof}

\begin{lemma}\label{lem:jtype-res}
Let $x$ be a nilpotent transformation on a vector space $V$, and let $F\supseteq\Im(x)$. Define
\begin{equation}\label{eq:bdim}b_j=\dim\left(F+\ker(x^j)\right), \quad j\geq 0.\end{equation}
Then $\J(x|_F)$ is obtained from $\J(x)$ by removing $(b_j-b_{j-1})$ boxes at the bottom of column $j$ for all $j\geq 1$. (Notice that these boxes form a vertical strip i.e. no two of them are in the same row)
\end{lemma}
\begin{proof}As in the previous proof, this follows directly from applying Lemma \ref{lem:general-res} to $U=0$, $W=F$ and using that $\dim(F+\ker(x^j))=\dim(F)+\dim(\ker(x^j))-\dim(F\cap \ker(x^j))$.
\end{proof}
\begin{lemma}\label{lem:ker-perp}
Let $V$ be a $2n$-dimensional symplectic space, and $x\in\End(V)$ such that $x\in \cS$. Then $\Im(x^j)=\ker(x^j)^\perp$ for all $j\geq 0$.
\end{lemma}
\begin{proof}
Let $w\in \Im(x^j)$, so $w=x^ju$ for some $u\in V$. If $z\in \ker(x^j)$, then 
$\langle w,z\rangle=\langle x^ju,z\rangle=\langle u,x^jz\rangle=0$
so $\Im(x^j)\subseteq\ker(x^j)^\perp$ and equality follows from the fact that they have the same dimension.
\end{proof}
\begin{lemma}
Let $V$ be a $2n$-dimensional symplectic space, and $x\in\cS\cap\cN$. 
Suppose $F\subseteq\ker(x)$, which is equivalent to $F^\perp\supseteq\Im(x)$, then
$$ \J(x|_{V/F})=\J(x|_{F^\perp}).$$
\end{lemma}
\begin{proof}Let $(a_j)_j$ be defined as in \eqref{eq:adim} for $F$, and $(b_j)_j$ be defined as in \eqref{eq:bdim} for the space $F^\perp$. Then $ (F\cap \im(x^j))^\perp=F^\perp+\ker(x^j)$ hence we have
\begin{align*}a_{j-1}-a_j&=\dim (F\cap \im(x^{j-1}))-\dim (F\cap \im(x^j)) \\
&=(2n-\dim(F^\perp+\ker(x^{j-1})))-(2n-\dim(F^\perp+\ker(x^j)))\\
&= b_j-b_{j-1}
\end{align*}
and the result then follows from Lemma \ref{lem:jtype-quot} and Lemma \ref{lem:jtype-res}.
\end{proof}
If $0\subseteq F\subseteq F^\perp\subseteq V$ is an isotropic subspace of $V$, and $(v,x)\in\fN$, in order to compute $\eJ\left(v+F,x|_{F^\perp/F}\right)$ we will need to use the following result due to Travkin \cite[Theorem 1 and Corollary 1]{Trav09} and Achar-Henderson \cite[Theorem 6.1]{AH}. 
\begin{theorem}\label{thm:t-ah}If $(v,x)\in\fN(W)$, then $\eJ(v,x)=(\mu,\nu)$ if and only if
$$\J\left(x|_W\right)=(\mu+\nu)\cup(\mu+\nu)=(\mu_1+\nu_1,\mu_1+\nu_1, \mu_2+\nu_2,\mu_2+\nu_2,...)$$
and
$$\J\left(x|_{W/\C[x]v}\right)=(\mu_1+\nu_1,\mu_2+\nu_1, \mu_2+\nu_2,\mu_3+\nu_2,...).$$
\end{theorem}
We want to apply this theorem to $W=F^\perp/F$, so we will need the next two lemmas, that tell us how to find $\J\left(x|_W\right)=\J\left(x|_{F^\perp/F}\right)$ and $\J\left(x|_{W/\C[x]v}\right)=\J\left(x|_{F^\perp/(F+\C[x]v)}\right)$.

\begin{lemma}\label{lem:FcapFperp}Let $V$ be a symplectic vector space, $x\in \cS\cap\cN$, and $0\subseteq F\subseteq F^\perp\subseteq V$ with $x(F)=0$ and $x(F^\perp)\subseteq F^\perp$. 
Let
$$ a_j'=\dim(F\cap x^{j}(F^\perp)), \quad j\geq 0$$
then $\J(x|_{F^\perp/F})$ is obtained from $\J(x|_{V/F})=\J(x|_{F^\perp})$ by removing $(a'_{j-1}-a'_j)$ boxes from the bottom of column $j$ for all $j\geq 1$.
\end{lemma}
\begin{proof}
We can either apply Lemma \ref{lem:jtype-quot} to the quotient $F^\perp\ronto F^\perp/F$ and look at the dimensions
$$ \dim(F\cap\im(x|_{F^\perp})^j)=\dim(F\cap x^j(F^\perp))$$
to get the result, or use Lemma \ref{lem:general-res} (with $U=F$, $W=F^\perp$).
\end{proof}
\begin{lemma}\label{lem:FcapFperpCv}
Let $V$ be a symplectic vector space, $(v,x)\in \fN$, and $0\subseteq F\subseteq F^\perp\subseteq V$ with $x(F)=0$, $x(F^\perp)\subseteq F^\perp$, $v\in F^\perp$.
Let
$$ b_j'=\dim((F+\C[x]v)\cap x^{j}(F^\perp)), \quad j\geq 0$$
then $\J(x|_{F^\perp/F+\C[x]v})$ is obtained from $\J(x|_{V/F})=\J(x|_{F^\perp})$ by removing $(b'_{j-1}-b'_j)$ boxes from the bottom of column $j$ for all $j\geq 1$.
\end{lemma}
\begin{proof}
This follows by applying Lemma \ref{lem:general-res} to $U=\F+\C[x]v\subseteq F^\perp=W$.
\end{proof}

\section{Computing $\J\left(x|_{F^\perp/F}\right)$, $\J\left(x|_{F^\perp/F+\C[x]v}\right)$ and $\eJ\left(v+F,x|_{F^\perp/F}\right)$}\label{sec:jordan-types}
For this section, we fix $x\in\fN\cap\cS$ and we apply the results of Section \ref{sec:jordan}. In particular we will look at the special case where $x^2=0$. We start with some general lemmas that we will need later.

\begin{definition}
For each $j\geq 0$, we define a map $\Im(x^j)\times \Im(x^j)\to \C$ denoted by $\llangle~,~\rrangle_j$, as follows. If $u,w\in \Im(x^j)$, let $z\in V$ such that $u=x^j z$, then
$$\llangle u,w\rrangle_j:=\langle z,w\rangle.$$
\end{definition}
\begin{lemma}
For each $j\geq 0$, the pairing $\llangle~,~\rrangle_j$ is well defined, and it is a nondegenerate skew-symmetric bilinear form on $\Im(x^j)$ (which is an even dimensional space).
\end{lemma}
\begin{proof}
To see that this pairing is well defined, let $u,w\in\Im(x^j)$ and suppose $z_1, z_2\in V$ are such that $u=x^jz_1=x^jz_2$. We also let $y\in V$ be such that $x^jy=w$. Then
$$\langle z_1,w\rangle=\langle z_1,x^jy\rangle=\langle x^j z_1,y\rangle= \langle x^j z_2,y\rangle=\langle z_2,x^jy\rangle=\langle z_2,w\rangle.$$
The fact that $\llangle~,~\rrangle_j$ is a skew-symmetric bilinear follows directly from the definition because if $u,w,z,y\in V$, with $x^jz=u$ and $x^jy=w$, then
$$\llangle u,w\rrangle_j = \langle z,w\rangle=-\langle w,z\rangle=-\langle x^jy,z\rangle=-\langle y,x^jz\rangle=-\langle y,u\rangle=-\llangle w,u\rrangle_j;$$
also, for $\alpha,\beta\in\C$ we have $x^j(\alpha z)=\alpha u$ and $x^j(\beta y)=\beta w$, so
$$\llangle \alpha u,\beta w\rrangle_j = \langle\alpha z,\beta w\rangle=\alpha\beta\langle z,w\rangle=\alpha\beta\llangle  u, w\rrangle_j.$$

Finally, suppose that there is $u\in \Im(x^j)$ such that $\llangle u,w\rrangle_j=0$ for all $w\in \Im(x^j)$, and let $z\in V$ with $x^jz=u$. Then $0=\llangle u,w\rrangle_j=\langle z,w\rangle$ for all $w\in\Im(x^j)$
so $z\in \left(\Im(x^j)\right)^\perp=\ker(x^j)$ (by Lemma \ref{lem:ker-perp}), hence $u=x^jz=0$. This proves that the form is indeed nondegenerate.
\end{proof}
In particular, $\llangle~,~\rrangle_0=\langle~,~\rangle$, and when $j=1$ we might omit the index and just denote $\llangle~,~\rrangle_1=\llangle~,~\rrangle$.
We denote the perpendicular with respect to the form $\llangle~,~\rrangle_j$ by $\pperp_j$. So, if $W\subset \Im(x^j)$, then 
$$W^{\pperp_j}=\{u\in\Im(x^j)~|~\llangle u,w\rrangle_j=0,\text{ for all }w\in W \}.$$
In particular $\pperp_0=\perp$ is the perpendicular with respect to $\langle~,~\rangle$, and when $j=1$ we write $\pperp=\pperp_1$.

\begin{lemma}\label{lem:F-rad-j}Let $F\subset V$ be a subspace, then for all $j\geq 0$, $x^j(F^\perp)=(F\cap \Im(x^j))^{\pperp_j}$.
\end{lemma}
\begin{proof}
Let $u\in x^j(F^\perp)$, then there is $z\in F^\perp$ such that $x^jz=u$, hence for all $w\in F\cap \Im(x^j)$ we have $\llangle u,w\rrangle_j=\langle z,w\rangle=0$, so $u\in (F\cap \Im(x^j))^{\pperp_j}$. Viceversa, if $u\in(F\cap \Im(x^j))^{\pperp_j}$, then $u\in\Im(x^j)$, so there is $z\in V$ with $u=x^jz$, and we have that for all $w\in F\cap \Im(x^j)$, $0=\llangle u,w\rrangle_j=\langle z,w\rangle$, so $z\in (F\cap \Im(x^j))^\perp=F^\perp +\Im(x^j)^\perp=F^\perp+\ker(x^j)$. This shows that 
$(F\cap \Im(x^j))^{\pperp_j}\subseteq x^j(F^\perp+\ker(x^j))=x^j(F^\perp)$ which concludes the proof.
\end{proof}
The next two lemmas will specifically be relevant in the case where $x^2=0$.
\begin{lemma}\label{lem:vImvnotinFk}
If $v\in\Im(x)$, $v\not\in F$, then $F\cap x(F^\perp)=(F+\C v)\cap x(F^\perp)$
if and only if
$$F\cap x(F^\perp)\not\subset (\C v)^{\pperp}.$$
\end{lemma}
\begin{proof}
Since $F\cap x(F^\perp)\subseteq (F+\C v)\cap x(F^\perp)$, we have equality exactly when the two spaces have the same dimension. Notice that $F\cap x(F^\perp)=F\cap \Im(x)\cap x(F^\perp)$ and, since $v\in\Im(x)$, 
$$(F+\C v)\cap x(F^\perp)=(F+\C v)\cap\Im(x)\cap x(F^\perp)=((F\cap \Im(x))+\C v)\cap x(F^\perp).$$
Hence $F\cap x(F^\perp)=(F+\C v)\cap x(F^\perp)$ if and only if
\begin{align}
&\dim\left(F\cap x(F^\perp)\right)=\dim\left((F+\C v)\cap x(F^\perp)\right)\notag \\
\iff & \dim\left(F\cap \Im(x)\cap x(F^\perp)\right)=\dim\left(((F\cap \Im(x))+\C v)\cap x(F^\perp)\right)\notag \\
\iff & \dim\left(\left(F\cap \Im(x)\cap x(F^\perp)\right)^{\pperp}\right)=\dim\left(\left(((F\cap \Im(x))+\C v)\cap x(F^\perp)\right)^{\pperp}\right) \notag\\
\iff & \dim\left(\left(F\cap \Im(x)\right)^{\pperp}+\left(x(F^\perp)\right)^{\pperp}\right)=\dim\left(\left(((F\cap \Im(x))+\C v)\right)^{\pperp}+ \left(x(F^\perp)\right)^{\pperp}\right)\notag \\
\iff & \dim\left(x(F^\perp)+\left(F\cap\Im(x)\right)\right)=\dim\left(\left((F\cap \Im(x))^{\pperp}\cap (\C v)^{\pperp}\right)+ \left(F\cap\Im(x)\right)\right)\notag \\
\iff & \dim\left(x(F^\perp)\right)+\dim\left(F\cap\Im(x)\right)-\dim\left(F\cap \Im(x)\cap x(F^\perp)\right)\notag \\
& =\dim\left(\left(x(F^\perp)\cap (\C v)^{\pperp}\right)+ \left(F\cap\Im(x)\right)\right)\notag \\
\iff & \dim\left(x(F^\perp)\right)+\dim\left(F\cap\Im(x)\right)-\dim\left(F\cap x(F^\perp)\right)\notag \\
& =\dim\left(x(F^\perp)\cap (\C v)^{\pperp}\right)+ \dim\left(F\cap\Im(x)\right)-\dim\left(x(F^\perp)\cap (\C v)^{\pperp}\cap F\cap\Im(x) \right)\notag \\
\iff & \dim\left(x(F^\perp)\right)+\dim\left(F\cap\Im(x)\right)-\dim\left(F\cap x(F^\perp)\right)\notag \\
& =\dim\left(x(F^\perp)\cap (\C v)^{\pperp}\right)+ \dim\left(F\cap\Im(x)\right)-\dim\left(x(F^\perp)\cap (\C v)^{\pperp}\cap F\cap\Im(x) \right)\notag \\
\iff & \dim\left(x(F^\perp)\right)-\dim\left(x(F^\perp)\cap (\C v)^{\pperp}\right)= \dim\left(F\cap x(F^\perp)\right)-\dim\left(F\cap x(F^\perp)\cap (\C v)^{\pperp}\right)\label{eq:diff-dim}
\end{align}
Now, since $(\C v)^{\pperp}$ is a codimension $1$ subspace of $\Im(x)$ we have that each of the two sides of the inequality are either equal to $0$ or $1$. Since $v\not\in F$, hence $v\not\in F\cap \Im(x)$, we have that
\begin{align*}
(F\cap\Im(x))+\C v & \neq F\cap\Im(x) \\
\left((F\cap\Im(x))+\C v\right)^{\pperp} & \neq \left(F\cap\Im(x)\right)^{\pperp} \\
\left((F\cap\Im(x)\right)^{\pperp}\cap \left(\C v\right)^{\pperp} & \neq \left(F\cap\Im(x)\right)^{\pperp}\\
x(F^\perp)\cap \left(\C v\right)^{\pperp}& \neq x(F^\perp),
\end{align*} hence the LHS of \eqref{eq:diff-dim} is $1$. Therefore, the RHS equals the LHS if and only if
$$F\cap x(F^\perp)\neq F\cap x(F^\perp)\cap (\C v)^{\pperp}$$
which is equivalent to $F\cap x(F^\perp)\not\subset (\C v)^{\pperp}$.
\end{proof}
\begin{lemma}\label{lem:vnotImvnotinFk}
If $v\not\in\Im(x)$, $v\not\in F$, then $F\cap x(F^\perp)=(F+\C v)\cap x(F^\perp)$
if and only if
$$(v+F)\cap x(F^\perp)=\emptyset.$$    
\end{lemma}
\begin{proof}
Notice that $v\not\in \Im(x)$, so in particular $v\not\in x(F^\perp)$. Since additionally $v\not\in F$, we have that $(F+\C v)\cap x(F^\perp)= F\cap x(F^\perp)$ if and only if for all $u\in F$, $u+v\not\in x(F^\perp)$.
\end{proof}
\subsection{Computing $\J\left(x|_{F_k^\perp/F_k}\right)$}We consider now any isotropic subspace $F_k\in\Gr^\perp_k(\ker(x)\cap(\C[x]v)^\perp)$, using the subscript to emphasize the dimension $k=\dim(F_k)$. We assume that $x^2=0$ and we separate two cases.
\subsubsection{$x=0$}
In this case, $\J(x)=(1^{2n})$, and $\Im(x^j)=0$ when $j>0$, hence, using Lemma \ref{lem:jtype-quot} we have $a_0=k$, $a_j=0$ for $j>0$ and so $\J(x|_{V/F_k})=\J(x|_{F_k^\perp})=(1^{2n-k})$. Further, using Lemma \ref{lem:FcapFperp} we have $a'_0=k$, $a'_j=0$ for $j>0$, and so $\J\left(x|_{F_k^\perp/F_k}\right)=(1^{2n-2k})$.

\subsubsection{$x\neq 0$, $x^2=0$} \label{section:xnonzero}
In this case, since $x\in\cS$, $\J(x)=\lambda\cup\lambda$, with $\lambda=(2^{n_2}1^{n_1})$, $2n_2+n_1=n$, we then have $\J(x)=(2^{2n_2}1^{2n_1})$. Let $k_2=\dim(F_k\cap \Im(x))$, $k_1=\dim\left(F_k/(F_k\cap \Im(x))\right)$, so $k=k_1+k_2$. Using Lemma \ref{lem:jtype-quot} we have $a_0=k=k_1+k_2$, $a_1=k_2$, $a_j=0$ for $j>1$. We remove $a_0-a_1=k_1$ boxes from the first column and $a_1-a_2=k_2$ boxes from the second column, which gives $\J(x|_{V/F_k})=\J(x|_{F_k^\perp})=(2^{2n_2-k_2}1^{2n_1-k_1+k_2})$. 

Then, to use Lemma \ref{lem:FcapFperp} we need to know $\dim(F_k\cap x(F^\perp_k))$ and, by Lemma \ref{lem:F-rad-j}, $F_k\cap x(F^\perp_k)$ is the radical of the bilinear form $\llangle~,~\rrangle$ restricted to $F_k\cap\Im(x)$. Therefore, $F_k\cap \Im(x)/F_k\cap x(F_k^\perp)$ is a symplectic space with form $\llangle~,~\rrangle$, hence even dimensional, say $\dim\left(F_k\cap \Im(x)/F_k\cap x(F_k^\perp)\right)=2h$. We then have from Lemma \ref{lem:FcapFperp} that $a'_0=k=k_1+k_2$, $a'_1=k_2-2h$, $a'_j=0$ for $j>1$, so we need to remove $k_1+2h$ boxes from the first column and $k_2-2h$ boxes from the second column which results in $\J\left(x|_{F_k^\perp/F_k}\right)=(2^{2n_2-2k_2+2h}1^{2n_1-2k_1+2k_2-4h})$. 

\subsection{Computing $\J\left(x|_{F_k^\perp/F_k+\C[x]v}\right)$} We consider $F_k$ as above, computing the Jordan type will then give us, using Theorem \ref{thm:t-ah}, $\eJ\left(v+F_k,x|_{F_k^\perp/F_k}\right)$. There are six cases to consider.

\subsubsection{$x=0$, $v=0$}\label{subsub:x0-v0}In this case, $\mu=\emptyset$, $\nu=\lambda=1^{n}$, so we have $\J\left(x|_{F_k^\perp/F_k+\C[x]v}\right)=\J\left(x|_{F_k^\perp/F_k}\right)=(1^{2n-2k})$ and hence $\eJ\left(v+F_k,x|_{F_k^\perp/F_k}\right)=(\emptyset,1^{n-k})$ by Theorem \ref{thm:t-ah}.
\subsubsection{$x=0$, $v\neq 0$}\label{subsub:x0-vnot0}In this case, $\mu=\lambda=1^n$, $\nu=\emptyset$.
\begin{enumerate}[(a)]
    \item If $v\in F_k$, then $b'_0=k$, $b'_j=0$ for $j>0$, so, by Lemma \ref{lem:FcapFperpCv}, $\J\left(x|_{F_k^\perp/F_k+\C[x]v}\right)=\J\left(x|_{F_k^\perp/F_k}\right)=(1^{2n-2k}),$ and $\eJ\left(v+F_k,x|_{F_k^\perp/F_k}\right)=(\emptyset,1^{n-k})$.
\item If $v\not\in F_k$, then $b'_0=k+1$, $b'_j=0$, so by Lemma \ref{lem:FcapFperpCv}, $\J\left(x|_{F_k^\perp/F_k+\C[x]v}\right)=(1^{2n-2k-1})$, and $\eJ\left(v+F_k,x|_{F_k^\perp/F_k}\right)=(1^{n-k},\emptyset)$.
\end{enumerate}
\subsubsection{$x\neq 0$, $x^2=0$, $v=0$}\label{subsub:xnot0-v0}In this case, $\mu=\emptyset$, $\nu=\lambda=2^{n_2}1^{n_1}$, so we have $\J\left(x|_{F_k^\perp/F_k+\C[x]v}\right)=\J\left(x|_{F_k^\perp/F_k}\right)=(2^{2n_2-2k_2+2h}1^{2n_1-2k_1+2k_2-4h})$ and hence $\eJ\left(v+F_k,x|_{F_k^\perp/F_k}\right)=(\emptyset, 2^{n_{2}-k_{2}+h}1^{n_{1}-k_{1}+k_{2}-2h})$ by Theorem \ref{thm:t-ah}.

\subsubsection{$x\neq 0$, $x^2=0$, $0\neq v\in \Im (x)$}\label{subsub:v-in-Im}In this case, $\mu=1^{n_2}$, $\nu=1^{n_2+n_1}$, with $n_2>0$.
\begin{enumerate}[(a)]
\item If $v\in F_k$, then $b'_j=a'_j$, for all $j\geq 0$, so, by Lemma \ref{lem:FcapFperpCv}, $\J\left(x|_{F_k^\perp/F_k+\C[x]v}\right)=\J\left(x|_{F_k^\perp/F_k}\right)=(2^{2n_2-2k_2+2h}1^{2n_1-2k_1+2k_2-4h})$, and $\eJ\left(v+F_k,x|_{F_k^\perp/F_k}\right)=(\emptyset,2^{n_{2}-k_{2}+h}1^{n_{1}-k_{1}+k_{2}-2h})$.
\item If $v\not\in F_k$, and $(F_k+\C v)\cap x(F_k^\perp)\supsetneq F_k\cap x(F_k^\perp)$ (by Lemma \ref{lem:vImvnotinFk} this means that $F_k\cap x(F_k^\perp)\subset (\C v)^{\pperp}$), then we have $b'_0=k+1=a'_0+1$, and $b'_1=k_2-2h+1=a'_1+1$. Hence $\J(x|_{F_{k}^{\perp}/F_{k} + \mathbb{C}[x]v})=(2^{2n_2-2k_2+2h-1}1^{2n_1-2k_1+2k_2-4h+1})$ and we now demonstrate how to obtain that $\eJ\left(v+F_k,x|_{F_k^\perp/F_k}\right)=(1^{n_2-k_2+h},1^{n_1+n_2-k_1-h})$, by using Theorem \ref{thm:t-ah}.

We want to find the exotic Jordan type $\eJ\left(v+F,x|_{F^\perp/F}\right)=(\mu',\nu')$. Since $\J\left(x|_{F_k^\perp/F_k}\right)=(2^{2n_2-2k_2+2h}1^{2n_1-2k_1+2k_2-4h})=(\mu'+\nu')\cup(\mu'+\nu')$, we have $\mu'_{n_1+n_2-k_1-h+1}+\nu'_{n_1+n_2-k_1-h+1}=0$, hence $\mu'_{n_1+n_2-k_1-h+1}=0$. Then, from $\J(x|_{F_{k}^{\perp}/F_{k} + \mathbb{C}[x]v})=(2^{2n_2-2k_2+2h-1}1^{2n_1-2k_1+2k_2-4h+1})=(\mu'_1+\nu'_1,\mu'_2+\nu'_1,\ldots)$ we get that $\mu'_{n_1+n_2-k_1-h+1}+\nu'_{n_1+n_2-k_1-h}=1$, hence $\nu'_{n_1+n_2-k_1-h}=1$. Inductively we have $\mu'_j=0$ for all $j>n_2-k_2+h$, and $\nu'_j=1$ for $n_2-k_2+h
\leq j\leq n_1+n_2-k_1-h+1$. Then, since $\mu'_{n_2-k_2+h}+\nu'_{n_2-k_2+h}=2$, we get $\mu'_{n_2-k_2+h}=1$ and then inductively $\mu'_j=\nu'_j=1$ for all $1\leq j\leq n_2-k_2+h$.
\item\label{item:3} If $v\not\in F_k$, and $(F_k+\C v)\cap x(F_k^\perp)= F_k\cap x(F_k^\perp)$ (by Lemma \ref{lem:vImvnotinFk} this means that $F_k\cap x(F_k^\perp)\not\subset (\C v)^{\pperp}$), then we have $b'_0=k+1=a'_0+1$, and $b'_1=k_2-2h=a'_1$. Hence $\J(x|_{F_{k}^{\perp}/F_{k} + \mathbb{C}[x]v})=(2^{2n_2-2k_2+2h}1^{2n_1-2k_1+2k_2-4h-1})$ and with Theorem \ref{thm:t-ah} and a similar inductive argument as above, which we will omit, we get $\eJ\left(v+F,x|_{F^\perp/F}\right)=(1^{n_1+n_2-k_1-h},1^{n_2-k_2+h})$.
\end{enumerate}

\subsubsection{$x\neq 0$, $x^2=0$, $v\in\ker(x)\setminus \Im (x)$} \label{subsub:v-in-ker} In this case, $\mu=1^{n_2+n_1}$, $\nu=1^{n_2}$, with $n_1,n_2>0$.
The various possibilities for $\eJ\left(v+F,x|_{F^\perp/F}\right)$ in this case are the same as in Section \ref{subsub:v-in-Im} and depend in the exact same way on whether $v\in F_k$ and whether $(F_k+\C v)\cap x(F_k^\perp)= F_k\cap x(F_k^\perp)$ (with the difference that we will use Lemma \ref{lem:vnotImvnotinFk} instead of Lemma \ref{lem:vImvnotinFk} to verify the condition).
\subsubsection{$x\neq 0$, $x^2=0$, $v\not\in\ker(x)$}\label{subsub:v-not-ker}In this case, $\mu=2^{n_2}1^{n_1}$, $\nu=\emptyset$, with $n_2>0$. Since $F_k\subset \ker(x)$, we always have that $v\not\in F_k$ and since $v\in F_n\subset F_k^\perp$, $xv\in x(F_k^\perp)$. Also, since $v=v_{1,2}+\cdots+v_{n_2,2}+v_{n_2+1,1}\cdots+ v_{n_1+n_2,1}$, and $xv=v_{1,1}+\cdots+v_{n_2,1}$, then
$$\dim (F_k+\C[x]v)=\dim(F_k+\C v+\C xv)=\begin{cases}k+1 & \text{ if }xv\in F_k \\ k+2 & \text{ if }xv\not\in F_k\end{cases}.$$
\begin{enumerate}[(a)]
\item If $xv\in F_k$, then $\dim((F_k+\C xv+\C v)\cap x(F_k^\perp))=\dim((F_k+\C v)\cap x(F_k^\perp))=\dim(F_k\cap x(F_k^\perp))=k_2-2h$ because for any $c\neq 0$, $(cv+\ker(x))\cap \Im(x)=\emptyset$.
It follows that in this case $b_0'=k+1$, $b_1'=k_2-2h$, so, like in \ref{subsub:v-in-Im}\eqref{item:3}, $\J(x|_{F_{k}^{\perp}/F_{k} + \mathbb{C}[x]v})=(2^{2n_2-2k_2+2h}1^{2n_1-2k_1+2k_2-4h-1})$ and $\eJ\left(v+F,x|_{F_k^\perp/F_k}\right)=(1^{n_1+n_2-k_1-h},1^{n_2-k_2+h})$.
    \item If $xv\not\in F_k$, then by the modular property of lattices of vector spaces we have $(F_k+\C xv+\C v)\cap x(F_k^\perp)=((F_k+\C v)\cap x(F_k^\perp))+\C xv$, hence
    \begin{align*}
        \dim (F_k+\C[x]v)&=\dim((F_k+\C xv+\C v)\cap x(F_k^\perp))\\
        &= \dim(((F_k+\C v)\cap x(F_k^\perp))+\C xv)\\
        &= \dim(F_k\cap x(F_k^\perp)+\C xv)\\
        &= k_2-2h+1.
    \end{align*}
It follows that in this case $b_0'=k+2$, $b_1'=k_2-2h+1$, so, $\J(x|_{F_{k}^{\perp}/F_{k} + \mathbb{C}[x]v})=(2^{2n_2-2k_2+2h-1}1^{2n_1-2k_1+2k_2-4h})$ and by Theorem \ref{thm:t-ah} and a similar inductive argument as above, we get $\eJ\left(v+F_k,x|_{F_k^\perp/F_k}\right)=(2^{n_2-k_2+h}1^{n_1-k_1+k_2-2h},\emptyset)$.

\end{enumerate}

\section{Proof of Conjecture \ref{prop:B-irred} when $x^2=0$}\label{sec:proof}

For notational simplicity, we let $\alpha_1=k$, so to prove the conjecture we want to describe the variety $$\cB_{((\mu',\nu'),\alpha')}^{((\mu,\nu),\alpha)}=\{F_{k}\in\Gr^\perp_k\left(\ker(x) \cap (\mathbb{C}[x]v)^{\perp}\right)\, | \, \eJ\left(v+F_{k},x|_{F_{k}^\perp/F_{k}}\right)=(\mu',\nu')\, \} $$
for all the possible cases of $(\mu,\nu)$ and $(\mu',\nu')$ that we found in Section \ref{sec:jordan-types}. More precisely, we want to show its irreducibility and compute its dimension, and to do that we will need some more fibre bundles. 

\subsection{More Fibre Bundles}\label{sec:iter-bundle}Given any $F_k\in \cB_{((\mu',\nu'),\alpha')}^{((\mu,\nu),\alpha)}$, we have a filtration $F_k\cap x(F_k^\perp)\subseteq F_k\cap\Im(x)\subseteq F_k$. 

\begin{definition}
For $\ell\geq 0$, we define two maps
$$\pi^\ell_1:\Gr^\perp_\ell(\ker(x)\cap(\C[x]v)^\perp)\to\coprod_{j\geq 0} \Gr_{j}(\ker(x)\cap(\C[x]v)^\perp\cap\Im(x))$$
 $$\pi^\ell_1(F)=F\cap\Im(x);$$
$$\pi^\ell_2:\Gr_{\ell}(\ker(x)\cap(\C[x]v)^\perp\cap\Im(x))\to\coprod_{j\geq 0} \Gr^{\pperp}_{j}(\ker(x)\cap(\C[x]v)^\perp\cap\Im(x)) $$
$$\pi^\ell_2(U)=U\cap U^{\pperp}.$$ 
\end{definition}
\begin{remark}
If $F_k\in\cB_{((\mu',\nu'),\alpha')}^{((\mu,\nu),\alpha)}$, with $\dim(F_k\cap\Im(x))=k_2$ as in Section \ref{section:xnonzero}, then by Lemma \ref{lem:F-rad-j} we have
$$\pi^{k_2}_2(\pi^k_1(F_k))=(F_k\cap\Im(x))\cap (F_k\cap \Im(x))^{\pperp}=F_k\cap x(F_k^\perp).$$
\end{remark}
If we have a subspace $U\subset \ker(x)\cap (\C[x]v)^\perp\cap \Im(x)$, $\dim(U)=u$, and $F\cap\Im(x)=U$, we can consider the following subspaces in the quotient space $\ker(x)\cap (\C[x]v)^\perp/U$: $\overline{F}=F/U$, and $\overline{\Im(x)}=((\ker(x)\cap(\C[x]v)^\perp\cap\Im(x))/U$. We then have 
\begin{align}\label{eq:pi1-inv}(\pi_1^\ell)^{-1}(U) &= \{F\in\Gr^\perp_\ell(\ker(x)\cap(\C[x]v)^\perp)~|~F\cap\Im(x)=U\}\\
&\cong\{\overline{F}\in\Gr^\perp_{\ell-u}((\ker(x)\cap(\C[x]v)^\perp)/U)~|~\overline{F}\cap\overline{\Im(x)}=0\} \notag
\end{align}
which is empty when $\ell-u+\dim((\ker(x)\cap(\C[x]v)^\perp\cap\Im(x))-u>\dim(\ker(x)\cap(\C[x]v)^\perp)-u$ and otherwise is an open dense subvariety of the isotropic Grassmannian $\Gr^\perp_{\ell-u}((\ker(x)\cap(\C[x]v)^\perp)/U)$. This shows that the map $\pi_1^\ell$ is a fibre bundle over any given $\Gr_{j}(\ker(x)\cap(\C[x]v)^\perp\cap\Im(x))$.

Similarly, if we have an isotropic subspace $S\subset S^{\pperp}$, with respect to $\llangle~,~\rrangle$ restricted to $\ker(x)\cap (\C[x]v)^\perp\cap \Im(x)$, with $\dim(S)=s$, and if $U\subseteq \ker(x)\cap(\C[x]v)^\perp\cap \Im(x)$, with $U\cap U^{\pperp}=S$, then $U^{\pperp}+U=S^{\pperp}$. We have then 
\begin{align}\label{eq:pi2-inv}(\pi_2^\ell)^{-1}(S) &= \{U\in\Gr_\ell(\ker(x)\cap(\C[x]v)^\perp\cap \Im(x))~|~U\cap U^{\pperp}=S\} \notag \\
&= \{U\in\Gr_\ell(\ker(x)\cap(\C[x]v)^\perp\cap \Im(x))~|~S\subseteq U\subseteq S^{\pperp}\}\\
&\cong\Gr_{\ell-s}(S^{\pperp}/S) \notag
\end{align}
This fibre is empty when $\ell-s>\dim((\ker(x)\cap(\C[x]v)^\perp\cap\Im(x)))-2s$ and otherwise is isomorphic to a Grassmannian. This shows that the map $\pi_2^\ell$ is a fibre bundle over any given $\Gr^{\pperp}_{j}(\ker(x)\cap(\C[x]v)^\perp\cap\Im(x))$.
\begin{remark}
In the computations of $d^{\alpha}_{(\mu,\nu)}$ it will be helpful to use the fact that $N(1^n)=n(n-1)/2$ and $N(2^{n_2}1^{n_1})=(n_2+n_1)(n_2+n_1-1)/2+n_2(n_2-1)/2$, also since $k=\alpha_1$, we have $\frac{1}{2} \sum_{i\geq 1} (\alpha_{i}^{2} - \alpha_{i})=\frac{1}{2}k(k-1)+\frac{1}{2} \sum_{i\geq 2} (\alpha_{i}^{2} - \alpha_{i})$.
\end{remark}

We now examine all the different possible cases for the proof of Conjecture \ref{prop:B-irred} when $x^2=0$.
\subsection{\underline{$(\mu,\nu) = (\emptyset, 1^{n})$}}In this case we have $x=0$ and $v=0$, so $\ker(x)\cap (\C[x]v)^\perp=V$. By Section \ref{subsub:x0-v0}, no matter the choice of $F_k\in\Gr^\perp_k(V)$, we have $(\mu',\nu')=(\emptyset,1^{n-k})$.

Therefore
$$ \cB_{((\emptyset,1^{n-k}),\alpha')}^{((\emptyset,1^n),\alpha)}=\Gr^\perp_k(\ker(x)\cap(\C[x]v)^{\perp})=\Gr^\perp_k(V)\simeq \Gr^\perp_{k,2n} $$
is an irreducible variety of dimension $k(2n-k)-\frac{1}{2}k(k-1)$ by Remark \ref{rem:grass-dim}. We also have
\begin{align*}
   d_{(\mu,\nu)}^{\alpha} - d_{(\mu',\nu')}^{\alpha'} &= 2N(1^{n}) + |1^n| - \frac{1}{2} \sum_{i\geq 1} (\alpha_{i}^{2} - \alpha_{i}) - \left(   2N(1^{n-k}) +|1^{n-k}| - \frac{1}{2} \sum_{i\geq 2}(\alpha_{i}^{2} - \alpha_{i}) \right)\\
   &= n(n-1)+n-(n-k)(n-k-1)-(n-k)-\frac{1}{2}k(k-1)\\
   &= n^2 -(n-k)^2-\frac{1}{2}k(k-1) \\
   &= k(2n - k)-\frac{1}{2}k(k-1)
\end{align*}
and the shape $(\emptyset,1^{n-k})$ is obtained from $(\emptyset,1^n)$ by removing $k$-boxes that are all in different rows.
\begin{example}We show the two diagrams with $n=5$ and $k=3$, as well as the boxes removed marked with an $\xx$.
$$(\emptyset,1^{n-k})=\left(\emptyset,\young(~,~)\right)\subset (\emptyset,1^{n})=\left(\emptyset,\young(~,~,~,~,~)\right),\qquad \left(\emptyset,\young(~,~,\xx,\xx,\xx)\right).$$
\end{example}

\subsection{\underline{$(\mu,\nu) = ( 1^{n},\emptyset)$}}In this case we have $x=0$ and $v\neq 0$, so $\ker(x)=V$ and $(\C[x]v)^\perp=(\C v)^\perp$ is a $2n-1$ dimensional subspace. By Section \ref{subsub:x0-vnot0}, we have $(\mu',\nu')=(\emptyset,1^{n-k})$ when $v\in F_k$, and $(\mu',\nu')=(1^{n-k},\emptyset)$ when $v\not\in F_k$. 

\subsubsection{}Suppose that $(\mu',\nu')=(\emptyset,1^{n-k})$, so that $v\in F_k$ (which implies $k\geq 1$), then $F_k/\C v$ is a $k-1$-dimensional isotropic subspace of $\ker(x)\cap(\C[x]v)^\perp/(\C v)=(\C v)^\perp/(\C v)$, which the restriction of $\langle,\rangle$ makes into a symplectic space of dimension $2n-2$. Hence

\begin{align*}
\cB_{((\emptyset,1^{n-k}),\alpha')}^{((1^n,\emptyset),\alpha)}&=\{F_{k} \subset \ker(x) \cap (\mathbb{C}[x]v)^{\perp}\, |~F_k\subset F_k^\perp,~v\in F_k\} \\
&\simeq \{F_{k}/(\C v) \subset (\C v)^\perp/(\C v)\, |~F_k/(\C v)\subset (F_k/(\C v))^\perp\}\\
&\simeq \Gr^\perp_{k-1}((\C v)^\perp/(\C v))\simeq \Gr^\perp_{k-1,2n-2}
\end{align*}
is an irreducible variety of dimension $(k-1)(2n-2-(k-1))-\frac{1}{2}(k-1)(k-2)=(k-1)(2n-k-1)-\frac{1}{2}(k-1)(k-2)$ by Remark \ref{rem:grass-dim}.

We then have
\begin{align*}
   & d_{(1^n,\emptyset)}^{\alpha} - d_{(\emptyset,1^{n-k})}^{\alpha'}-\dim \cB_{((\emptyset,1^{n-k}),\alpha')}^{((1^n,\emptyset),\alpha)}\\
   &= 2N(1^{n}) +0- \frac{1}{2} \sum_{i\geq 1} (\alpha_{i}^{2} - \alpha_{i}) -\left( 2N(1^{n-k}) +|1^{n-k}| - \frac{1}{2} \sum_{i\geq 2}(\alpha_{i}^{2} - \alpha_{i}) \right)\\
   &~ -\left((k-1)(2n-k-1)-\frac{1}{2}(k-1)(k-2)\right)\\
   &= n(n-1)-(n-k)(n-k-1)-(n-k)-\frac{1}{2}k(k-1)-(k-1)(2n-k-1)+\frac{1}{2}(k-1)(k-2)\\
   &= n-k.
\end{align*}
This shows that if $1\leq k<n$, $\dim \cB_{((\emptyset,1^{n-k}),\alpha')}^{((1^n,\emptyset),\alpha)}<d_{(1^n,\emptyset)}^{\alpha} - d_{(\emptyset,1^{n-k})}^{\alpha'}$ and in fact the shape $(\emptyset,1^{n-k})$ is not contained in $(1^n,\emptyset)$.
In the case when $k=n$ (which means that $m=1$ and there is a single space in the flag) then 
$$\dim\cB_{((\emptyset,1^{n-k}),\alpha')}^{((1^n,\emptyset),\alpha)}=\dim\cB_{((\emptyset,\emptyset),\alpha')}^{((1^n,\emptyset),\alpha)}= d_{(1^n,\emptyset)}^{\alpha} - d_{(\emptyset,\emptyset)}^{\alpha'}$$ 
and $(\emptyset,1^{n-k})=(\emptyset,\emptyset)$ is indeed obtained from $(1^n,\emptyset)$ by removing $k$ boxes in different rows.
\begin{example}We show the two diagrams with $n=5$ and $k=3$, where the diagrams are not nested, and the case with $n=k=5$, with the boxes removed marked with an $\xx$.
$$(\emptyset,1^{n-k})=\left(\emptyset,\young(~,~)\right)\not\subset (1^{n},\emptyset)=\left(\young(~,~,~,~,~),\emptyset\right);\qquad (\emptyset,\emptyset)\subset(1^{n},\emptyset)=\left(\young(~,~,~,~,~),\emptyset\right),\quad\left(\young(\xx,\xx,\xx,\xx,\xx),\emptyset\right).$$
\end{example}
\subsubsection{}Suppose that $(\mu',\nu')=(1^{n-k},\emptyset)$, so that $v\not\in F_k$, which implies $k<n$. Then $F_k$ is a $k$-dimensional isotropic subspace of $\ker(x)\cap (\C[x]v)^\perp=(\C v)^\perp$ which is a $2n-1$-dimensional space with the restriction of the form $\langle,\rangle$ having rank $2n-2$, so $r=n-1$. We then have 
\begin{align*}
\cB_{((1^{n-k},\emptyset),\alpha')}^{((1^n,\emptyset),\alpha)}&=\{F_{k} \subset (\mathbb{C}v)^{\perp}\, |~F_k\subset F_k^\perp,~v\not\in F_k\}    
\end{align*}
which is an open dense subvariety of
$$\Gr_{k}^{\perp}\left((\C v)^\perp\right)\simeq \Gr_{k,2n-1}^{n-1}$$
hence it is irreducible of dimension $k(2n-1-k)-\frac{1}{2}k(k-1)$ by Lemma \ref{lem:dim-deg-grass} (there is no third term because $k\leq n-1$).

In this case, we have:
\begin{align*}
d_{(1^n,\emptyset)}^{\alpha} - d_{(1^{n-k},\emptyset)}^{\alpha'} &= 2N(1^n) - \frac{1}{2} \sum_{i\geq 1} (\alpha_{i}^{2} - \alpha_{i})  - (2N(1^{n-k}) - \frac{1}{2} \sum_{i\geq 2} (\alpha_{i}^{2} - \alpha_{i})) \\
&= n(n-1) - \frac{1}{2}(k^{2} - k)  - (n-k)(n - k-1) \\
&= k(2n - k - 1) -  \frac{1}{2}(k^{2} - k) 
\end{align*}
which agrees with the dimension of the variety and the shape $(1^{n-k},\emptyset)$ is obtained from $(1^n,\emptyset)$ by removing $k$-boxes that are all in different rows.
\begin{example}We show the two diagrams with $n=5$ and $k=3$, as well as the boxes removed marked with an $\xx$.
$$(1^{n-k},\emptyset)=\left(\young(~,~),\emptyset\right)\subset (1^{n},\emptyset)=\left(\young(~,~,~,~,~),\emptyset\right),\qquad \left(\young(~,~,\xx,\xx,\xx),\emptyset\right).$$
\end{example}
\subsection{\underline{$\mu=\emptyset$, $\nu=2^{n_{2}}1^{n_{1}}$}}\label{sec:mu-empty} In this case we have $x\neq 0$, $x^2=0$ and $v= 0$, so $(\C[x]v)^\perp=V$ and $\ker(x)$ is a $2n_1+2n_2$-dimensional space where the restriction of $\langle,\rangle$ has rank $2n_1$ ($r=n_1$). We define $k_1$, $k_2$, and $h$ as in Section \ref{section:xnonzero}, by Section \ref{subsub:xnot0-v0} we have that $(\mu',\nu')=(\emptyset, 2^{n_{2}-k_{2}+h}1^{n_{1}-k_{1}+k_{2}-2h})$. To describe the variety $\cB_{((\mu',\nu'),\alpha')}^{((\mu,\nu),\alpha)}$ we use the ideas from Section \ref{sec:iter-bundle}. We start by describing the variety 
$$ \mathcal{X}(k_2,h):=\{F_k\in\Gr_k^\perp(\ker(x))~|~\dim(F_k\cap \Im(x))=k_2,\dim(F_k\cap x(F_k^\perp))=k_2-2h\}$$
\begin{lemma}\label{lem:ineq}
For $0\leq k\leq n_1+2n_2$, the variety $\mathcal{X}(k_2,h)$ is nonempty if and only if the following inequalities are satisfied:
\begin{equation}\label{eq:ineq}
\max\{0,k-n_1\}\leq k_2\leq\min\{k,2n_2\};\qquad \max\{0,k_2-n_2\}\leq h\leq \frac{k_2}{2}.    
\end{equation}
\end{lemma}
\begin{proof}If the variety is nonempty, then for all $F_k\in\mathcal{X}(k_2,h)$ we have
$$0\leq\dim(F_k\cap \Im(x))=k_2\leq \dim(\Im(x))=2n_2,$$ 
$$k_1=\dim(F_k/F_k\cap\Im(x))\geq 0, \text{ so }k_2=k-k_1\leq k,$$ $$F_k/(F_k\cap\Im(x))\in\Gr_{k_1}^{\perp n_1}(\ker(x)/F_k\cap\Im(x)), \text{ so }k_1\leq n_1, \text{ and }k_2=k-k_1\geq k-n_1,$$ 
$$0\leq \dim(F_k\cap\Im(x)/F_k\cap x(F_k^\perp))=2h\leq\dim(F_k\cap\Im(x))=k_2,$$ 
Since $F_k\cap\Im(x)\subset (F_k\cap x(F_k^\perp))^{\pperp}$, we have $k_2\leq 2n_2-(k_2-2h)$, which implies 
$$0\leq n_2-k_2+h,\text{ so } h\geq k_2-n_2.$$ 
Viceversa, if all the inequalities are satisfied, we can define $F_k=\Span\{u_1,\ldots,u_k\}$. Here we choose $u_1,\ldots,u_{k_2-2h}$ to be linearly independent vectors in $\Im(x)$ such that $\llangle u_i,u_j\rrangle=0$ for all $1\leq i,j,\leq k_2-2h$, then we choose $u_{k_2-2h},\ldots,u_{k_2}$ to be linearly independent vectors in $(\Span\{u_1,\ldots,u_{k_2-2h}\}^{\pperp})\setminus\Span\{u_1,\ldots,u_{k_2-2h}\}$, and finally we choose $u_{k_2+1},\ldots,u_k$ to be linearly independent vectors that are not in $\Im(x)$, such that $\langle u_i,u_j\rangle=0$ for all $k_2+1\leq i,j,\leq k$. Then $F_k\in\mathcal{X}(k_2,h)\neq\emptyset.$
\end{proof}
Observe that
\begin{align*}\mathcal{X}(k_2,h)&:=\{F_k\in\Gr_k^\perp(\ker(x))~|~\dim(F_k\cap \Im(x))=k_2,\dim(F_k\cap x(F_k^\perp))=k_2-2h\}\\
&= (\pi_1^k)^{-1}(\pi_2^{k_2})^{-1}\left(\Gr_{k_2-2h}^{\pperp}(\ker(x)\cap(\C[x]v)^\perp\cap\Im(x))\right)\\
&= (\pi_1^k)^{-1}(\pi_2^{k_2})^{-1}\left(\Gr_{k_2-2h}^{\pperp}(\Im(x))\right).
\end{align*}

The base $\Gr_{k_2-2h}^{\pperp}(\Im(x))$ is an isotropic Grassmannian, which (since $\dim(\Im(x))=2n_2$), is irreducible of dimension 
$$(k_{2} - 2h)(2n_{2} - k_{2} +2h) - \tfrac{1}{2} (k_{2}-2h) (k_{2}-2h-1).$$
Over every point $S\in\Gr_{k_2-2h}^{\pperp}(\Im(x))$, we have $(\pi_2^{k_2})^{-1}(S)\simeq \Gr_{2h}(S^{\pperp}/S)$ by \eqref{eq:pi2-inv}, so each fibre is irreducible of dimension
$$ 2h(2n_{2} - 2k_{2}+2h)$$
hence $(\pi_2^{k_2})^{-1}\left(\Gr_{k_2-2h}^{\pperp}(\Im(x))\right)$ is irreducible, being a fibre bundle with irreducible fibres over an irreducible base.
Again, for each $U\in(\pi_2^{k_2})^{-1}\left(\Gr_{k_2-2h}^{\pperp}(\Im(x))\right) $, we consider the fibre $(\pi_1^k)^{-1}(U)$ which by \eqref{eq:pi1-inv} is isomorphic to an open dense subvariety of $\Gr^\perp_{k_1}(\ker(x)/U)$. Notice that $\ker(x)/U$ is a vector space of dimension $2n_1+2n_2-k_2$ and the restriction of $\langle,\rangle$ has rank $2n_1$, hence each fibre $(\pi_1^k)^{-1}(U)$ is an irreducible subvariety of dimension 
$$\dim\Gr_{k_1,2n_1+2n_2-k_2}^{\perp n_1}=k_{1}(2n_{2}+2n_{1} - k_{2} - k_{1}) - \tfrac{1}{2}k_{1}(k_{1}-1).$$
In conclusion $\mathcal{X}(k_2,h)$ is an irreducible variety of dimension
\begin{align}\label{eq:dim-var-xnot0v0}
    &(k_{2} - 2h)(2n_{2} - k_{2} +2h) - \tfrac{1}{2} (k_{2}-2h) (k_{2}-2h-1)+2h(2n_{2} - 2k_{2}+2h)+\\
    &+k_{1}(2n_{2}+2n_{1} - k_{2} - k_{1}) - \tfrac{1}{2}k_{1}(k_{1}-1). \notag
\end{align}

For $(\mu,\nu)=(\emptyset,2^{n_2}1^{n_1})$, $(\mu',\nu')=(\emptyset,2^{n_{2}-k_{2}+h}1^{n_{1}-k_{1}+k_{2}-2h})$ we compute the difference : 

\begin{equation}\label{eq:diff-dd'-dim}
	\begin{split}
    d_{(\mu,\nu)}^{\alpha} - d_{(\mu',\nu')}^{\alpha'} -\dim\mathcal{X}(k_2,h)&=2N(2^{n_2}1^{n_1})+2n_2+n_1- \frac{1}{2} \sum_{i\geq 1} (\alpha_{i}^{2} - \alpha_{i})-2N(2^{n_{2}-k_{2}+h}1^{n_{1}-k_{1}+k_{2}-2h})\\
    &\quad -(2(n_{2}-k_{2}+h)+n_{1}-k_{1}+k_{2}-2h)+ \frac{1}{2} \sum_{i\geq 2} (\alpha_{i}^{2} - \alpha_{i})-\eqref{eq:dim-var-xnot0v0}\\
    & = (n_{1}+ n_{2})(n_{1} + n_{2} - 1) + n_{2}(n_{2}-1) + 2n_{2} + n_{1}- \tfrac{1}{2}k(k-1) \\
    &\quad -(n_{1}+ n_{2}-k_{1}-h) (n_{1}+ n_{2}-k_{1}-h - 1) \\
    &\quad - (n_{2}-k_{2}+h)(n_{2}-k_{2}+h-1) \\
    & \quad - (2(n_{2}-k_{2}+h) + n_{1}-k_{1}+k_{2} - 2h)-\eqref{eq:dim-var-xnot0v0}\\
    &= h(2n_{1} - 2k_{1} + 1).
    \end{split}
\end{equation}
Since $h\geq 0$ and $n_1\geq k_1$, this difference is always $\geq 0$ and is equal to zero if and only if $h=0$.

Notice that for fixed $0\leq n_1'\leq n_1+n_2$, $0\leq n_2'\leq n_2$ such that $n_1'+2n_2'+k=n_1+2n_2$, we have 
$$\cB_{((\emptyset,2^{n_{2}'}1^{n_{1}'}),\alpha')}^{((\emptyset,2^{n_2}1^{n_1}),\alpha)}=\coprod_{\substack{\max\{0,k-n_1\}\leq k_2\leq\min\{k,2n_2\} \\ \max\{0,k_2-n_2\}\leq h\leq \frac{k_2}{2} \\ k_2-h=n_2-n_2'}}\mathcal{X}(k_2,h)=\coprod_{\substack{\max\{0,n_2-n_1'-n_2'\}\leq h \\
h\leq\min\{k-(n_2-n_2'), n_2-n_2'\}}}\mathcal{X}(h+n_2-n_2',h).
$$
Where the inequalities follow from \eqref{eq:ineq}, together with the fact that $n'_2=n_2-k_2+h$ and $n'_1=n_1-k_1+k_2-2h$ by Section \ref{subsub:xnot0-v0}. 
\begin{lemma}\label{lem:Xinclosure}Suppose that $\cX(k_2,h)\neq\emptyset\neq\cX(k_2+1,h+1)$, then $\cX(k_2+1,h+1)\subset\overline{\cX(k_2,h)}$.
\end{lemma}
\begin{proof}
Let $F_k\in\cX(k_2+1,h+1)\neq\emptyset$, such that $F_k=\Span\{u_1,\ldots,u_k\}$, with $F_k\cap x(F_k^\perp)=\Span\{u_1,\ldots,u_{k_2-2h-1}\}$, $F_k\cap\Im(x)=\Span\{u_1,\ldots,u_{k_2+1}\}$. Since $\cX(k_2,h)\neq\emptyset$, we have $k-n_1\leq k_2$, so $k-(k_2+1)<n_1$, it follows that there exists a vector $w\not\in\Im(x)$, $w\in F_k^\perp\not\in F_k$. Also, $\cX(k_2,h)\neq\emptyset$ implies that $h\geq 0$, hence $2h+2>0$ and $u_{k_2+1}\in F_k\cap\Im(x)\setminus F_k\cap x(F_k^\perp)$. For all $\alpha\in\C$, define a new space $F^\alpha_k$ to be spanned by the same vectors as $F_k$ except that $u_{k_2+1}$ is replaced by $u_{k_2+1}+\alpha w$. Then for all $\alpha\neq 0$ we have $F^\alpha_k\in\cX(k_2,h)$, which shows that $F_k=F^0_k\in\overline{\cX(k_2,h)}$.
\end{proof}
By \eqref{eq:diff-dd'-dim}, since $h(2n_{1} - 2k_{1} + 1)=h(2n_1+2n_2-2n_2'-2k+2h+1)$, and by Lemma \ref{lem:Xinclosure}, we have that for $h'> h$, $\mathcal{X}(h'+n_2-n_2',h')$ is contained in the closure of $\mathcal{X}(h+n_2-n_2',h)$ as a positive codimension subvariety. Therefore $\cB_{((\mu',\nu'),\alpha')}^{((\mu,\nu),\alpha)}=\overline{\mathcal{X}(h+n_2-n_2',h)}$, for the minimal possible value of $h$, i.e. $h=\max\{0,n_2-n_1'-n_2'\}$. This proves that $\cB_{((\mu',\nu'),\alpha')}^{((\mu,\nu),\alpha)}$ is an irreducible variety of dimension $\dim\mathcal{X}(h+n_2-n_2',h)$ with $h=\max\{0,n_2-n_1'-n_2'\}$.

Now, the minimal $h$ is $h=0$ if and only if $n_2\leq n_1'+n_2'$ if and only if the shape $(\emptyset,2^{n_2'}1^{n_1'})$ is obtained from $(\emptyset,2^{n_2}1^{n_1})$ by removing boxes in a vertical strip. In conclusion, 
$\dim\cB_{((\mu',\nu'),\alpha')}^{((\mu,\nu),\alpha)}=d_{(\mu,\nu)}^{\alpha} - d_{(\mu',\nu')}^{\alpha'}$ exactly when the shape is obtained by removing boxes from different rows and is strictly lower otherwise.
\begin{example}Here $n_1=2$, $n_2=3$, $k=4$. In the first example $n'_1+n_2'=0+2<3=n_2$ so some boxes are removed from the same row, while in the second example $n'_1+n'_2=2+1=3\geq n_2$ so the boxes are removed from different rows. 
$$(\emptyset,2^{n'_2}1^{n'_1})=\left(\emptyset,\young(~~,~~)\right)\subset (\emptyset,2^{n_2}1^{n_1})=\left(\emptyset,\young(~~,~~,~~,~,~)\right),\quad \left(\emptyset,\young(~~,~~,\xx\xx,\xx,\xx)\right);$$
$$(\emptyset,2^{n'_2}1^{n'_1})=\left(\emptyset,\young(~~,~,~)\right)\subset (\emptyset,2^{n_2}1^{n_1})=\left(\emptyset,\young(~~,~~,~~,~,~)\right),\quad \left(\emptyset,\young(~~,~\xx,~\xx,\xx,\xx)\right).$$
\end{example}

\subsection{\underline{$\mu=1^{n_{2}}$, $\nu=1^{n_{1} + n_{2}}$}} 
In this case we have $x\neq 0$, $x^2=0$ and $0\neq v\in\Im(x)$, so $\ker(x)\subset (\C[x]v)^\perp$ and $\ker(x)\cap (\C[x]v)^\perp=\ker(x)$ is a $2n_1+2n_2$-dimensional space where the restriction of $\langle,\rangle$ has rank $2n_1$ ($r=n_1$). We define $k_1$, $k_2$ and $h$ as in Section \ref{section:xnonzero}. According to Section \ref{subsub:v-in-Im}, there are three possibilities for $(\mu',\nu')$, which we now examine separately.

\subsubsection{}Suppose that $(\mu',\nu')=(\emptyset,2^{n_2'}1^{n_1'})=(\emptyset,2^{n_2-k_2+h}1^{n_1-k_1+k_2-2h})$. Then by Section \ref{subsub:v-in-Im}, we have $v\in F_k\cap\Im(x)$, which implies that $k_2\geq 1$. Notice that there are two possibilities, either $v\in F_k\cap x(F_k^\perp)$, or $v\in (F_k\cap\Im(x))\setminus (F_k\cap x(F_k^\perp))$. We define the following varieties
$$\mathcal{X}^1(k_2,h)=\{F_k\in\Gr_k^\perp(\ker(x))~|~\dim(F_k\cap \Im(x))=k_2,\dim(F_k\cap x(F_k^\perp))=k_2-2h,v\in F_k\cap x(F_k^\perp) \} $$
\begin{multline*}\mathcal{Y}(k_2,h)=\{F_k\in\Gr_k^\perp(\ker(x))~|~\dim(F_k\cap \Im(x))=k_2, \\ \dim(F_k\cap x(F_k^\perp))=k_2-2h,v\in (F_k\cap\Im(x))\setminus (F_k\cap x(F_k^\perp)) \}.\end{multline*}
If $v\in F_k\cap x(F_k^\perp)$, then $k_2-2h\geq 1$, and $(F_k\cap x(F_k^\perp))/\C v\in \Gr_{k_{2} - 2h -1}^{\pperp}(\mathbb{C}(v)^{\pperp}/\mathbb{C}v).$
It follows that we have a description as an iterated vector bundle
$$\mathcal{X}^1(k_2,h)=(\pi_1^k)^{-1}(\pi_2^{k_2})^{-1}\left(\Gr_{k_2-2h-1}^{\pperp}(\mathbb{C}(v)^{\pperp}/\mathbb{C}v)\right).$$
The base is an isotropic Grassmannian, hence an irreducible variety of dimension
\begin{equation} \label{eq:nulonger1}
(k_{2} - 2h-1)(2n_{2} - k_{2}  + 2h -1) - \tfrac12(k_{2}-2h-1)(k_{2}-2h-2).
\end{equation}
Over every point $S\in\Gr_{k_2-2h-1}^{\pperp}(\mathbb{C}(v)^{\pperp}/\mathbb{C}v)$, we have $(\pi_2^{k_2})^{-1}(S)\simeq \Gr_{2h}(S^{\pperp}/S)$ by \eqref{eq:pi2-inv}, so each fibre is irreducible of dimension
$$ 2h(2n_{2} - 2k_{2}+2h).$$
Finally over a point $U\in(\pi_2^{k_2})^{-1}\left(\Gr_{k_2-2h-1}^{\pperp}(\mathbb{C}(v)^{\pperp}/\mathbb{C}v)\right) $, we have that
$(\pi_1^k)^{-1}(U)$ is isomorphic to an open subvariety of $\Gr^\perp_{k_1}(\ker(x)/U)$. Exactly as in the previous section, then, each fibre $(\pi_1^k)^{-1}(U)$ is an irreducible subvariety of dimension 
$$\dim\Gr_{k_1,2n_1+2n_2-k_2}^{\perp n_1}=k_{1}(2n_{2}+2n_{1} - k_{2} - k_{1}) - \tfrac{1}{2}k_{1}(k_{1}-1).$$
In conclusion $\mathcal{X}^1(k_2,h)$ is an irreducible variety of dimension
\begin{align}\label{eq:dim-var-xnot0-vRad}
    &(k_{2} - 2h-1)(2n_{2} - k_{2}  + 2h -1) - \tfrac12(k_{2}-2h-1)(k_{2}-2h-2)+2h(2n_{2} - 2k_{2}+2h)+\\
    &+k_{1}(2n_{2}+2n_{1} - k_{2} - k_{1}) - \tfrac{1}{2}k_{1}(k_{1}-1). \notag
\end{align}
For $(\mu,\nu)=(1^{n_2},1^{n_1+n_2})$, $(\mu',\nu')=(\emptyset,2^{n_{2}-k_{2}+h}1^{n_{1}-k_{1}+k_{2}-2h})$ we compute the difference : 

\begin{equation}\label{eq:diff-dd'-dimX}
	\begin{split}
    d_{(\mu,\nu)}^{\alpha} - d_{(\mu',\nu')}^{\alpha'} -\dim\mathcal{X}^1(k_2,h)&=2N(2^{n_2}1^{n_1})+n_2+n_1- \frac{1}{2} \sum_{i\geq 1} (\alpha_{i}^{2} - \alpha_{i})-2N(2^{n_{2}-k_{2}+h}1^{n_{1}-k_{1}+k_{2}-2h})\\
    &\quad -(2(n_{2}-k_{2}+h)+n_{1}-k_{1}+k_{2}-2h)+ \frac{1}{2} \sum_{i\geq 2} (\alpha_{i}^{2} - \alpha_{i})-\eqref{eq:dim-var-xnot0-vRad}\\
    & = (n_{1}+ n_{2})(n_{1} + n_{2} - 1) + n_{2}(n_{2}-1) + n_{2} + n_{1}- \tfrac{1}{2}k(k-1) \\
    &\quad -(n_{1}+ n_{2}-k_{1}-h) (n_{1}+ n_{2}-k_{1}-h - 1) \\
    &\quad - (n_{2}-k_{2}+h)(n_{2}-k_{2}+h-1) \\
    & \quad - (2(n_{2}-k_{2}+h) + n_{1}-k_{1}+k_{2} - 2h)-\eqref{eq:dim-var-xnot0-vRad}\\
    &= 2h(n_{1} - k_{1} +1 )+n_2-k_2+h.
    \end{split}
\end{equation}

Since $h\geq 0$, $n_1\geq k_1$, and $n_2-k_2+h\geq 0$, this difference is always $\geq 0$ and is equal to zero if and only if $h=0$ and $n_2-k_2+h=0$ (which implies that $n_2=k_2$).

Now we want to describe $\mathcal{Y}(k_2,h)$, notice that if $v\in (F_k\cap\Im(x))\setminus (F_k\cap x(F_k^\perp))$, that means $h\geq 1$. Since $F_k\cap x(F_k^{\perp})=(F_k\cap \Im(x))^{\pperp}$, if $v\in (F_{k}\cap \Im(x))\setminus (F_k\cap x(F_k^\perp))$, then $F_k\cap x(F_k^{\perp})\subseteq (\C v)^{\pperp}$, with $\C v\not\subset F_k\cap x(F_k^\perp) $. Also the bilinear form $\llangle~,~\rrangle$ restricted to the $(2n_2 - 1)$-dimensional space $(\C v)^{\pperp}\subset \Im(x)$ has rank $2n_2-2$ ($r=n_2-1$). It follows that $F_k\cap x(F_k^{\perp})\in Z:= \{S\in\Gr^{\pperp}_{k_2-2h}((\C v)^{\pperp})~|~S\cap \C v=0\}$ which is an open dense subvariety of  $\Gr_{k_2-2h}^{\pperp}((\C v)^{\pperp})$, hence irreducible of dimension
\begin{equation*} 
\dim\Gr_{k_2-2h,2n_2-1}^{\perp n_2-1}=(k_{2}-2h)(2n_{2}-1-k_2+2h)-\tfrac{1}{2}(k_{2}-2h)(k_{2}-2h-1).
\end{equation*}
(There is no third term because $k-r=k_2-2h-(n_2-1)=-(n_2-k_2+h)-h+1\leq 0$).
  It follows that
$$\mathcal{Y}(k_2,h)=\{F_k\in (\pi_1^k)^{-1}(\pi_2^{k_2})^{-1}(Z)~|~v\in \pi_1^k(F_k)\}.$$
Over every point $S\in Z$, we consider only the points $U\in(\pi_2^{k_2})^{-1}(S)$ such that $v\in U$, i.e. we are choosing $U/(S+\C v)\in \Gr_{2h-1}(S^{\pperp}/(S+\C v))$, so each fibre is irreducible of dimension
$$\dim\Gr_{2h-1,2n_2-2(k_2-2h)-1}= (2h-1)(2n_{2}-2(k_{2}-2h)-1-(2h-1))=(2h-1)(2n_2-2k_2+2h).$$
Finally, over a point $U\in(\pi_2^{k_2})^{-1}\left(Z\right)$, (with $v\in U$) we have that
$(\pi_1^k)^{-1}(U)$ is isomorphic to an open dense subvariety of $\Gr^\perp_{k_1}(\ker(x)/U)$ so, as in previous sections, it is irreducible of dimension 
$$\dim\Gr_{k_1,2n_1+2n_2-k_2}^{\perp n_1}=k_{1}(2n_{2}+2n_{1} - k_{2} - k_{1}) - \tfrac{1}{2}k_{1}(k_{1}-1).$$
In conclusion $\mathcal{Y}(k_2,h)$ is an irreducible variety of dimension
\begin{align}\label{eq:dim-var-xnot0-vnotRad}
    &(k_{2}-2h)(2n_{2}-1-k_2+2h)-\tfrac{1}{2}(k_{2}-2h)(k_{2}-2h-1)+(2h-1)(2n_{2} - 2k_{2}+2h)+\\
    &+k_{1}(2n_{2}+2n_{1} - k_{2} - k_{1}) - \tfrac{1}{2}k_{1}(k_{1}-1). \notag
\end{align}
For $(\mu,\nu)=(1^{n_2},1^{n_1+n_2})$, $(\mu',\nu')=(\emptyset,2^{n_{2}-k_{2}+h}1^{n_{1}-k_{1}+k_{2}-2h})$ we compute the difference : 

\begin{equation}\label{eq:diff-dd'-dimY}
	\begin{split}
    d_{(\mu,\nu)}^{\alpha} - d_{(\mu',\nu')}^{\alpha'} -\dim\mathcal{Y}(k_2,h)&=2N(2^{n_2}1^{n_1})+n_2+n_1- \frac{1}{2} \sum_{i\geq 1} (\alpha_{i}^{2} - \alpha_{i})-2N(2^{n_{2}-k_{2}+h}1^{n_{1}-k_{1}+k_{2}-2h})\\
    &\quad -(2(n_{2}-k_{2}+h)+n_{1}-k_{1}+k_{2}-2h)+ \frac{1}{2} \sum_{i\geq 2} (\alpha_{i}^{2} - \alpha_{i})-\eqref{eq:dim-var-xnot0-vnotRad}\\
    & = (n_{1}+ n_{2})(n_{1} + n_{2} - 1) + n_{2}(n_{2}-1) + n_{2} + n_{1}- \tfrac{1}{2}k(k-1) \\
    &\quad -(n_{1}+ n_{2}-k_{1}-h) (n_{1}+ n_{2}-k_{1}-h - 1) \\
    &\quad - (n_{2}-k_{2}+h)(n_{2}-k_{2}+h-1) \\
    & \quad - (2(n_{2}-k_{2}+h) + n_{1}-k_{1}+k_{2} - 2h)-\eqref{eq:dim-var-xnot0-vnotRad}\\
    &= 2h(n_{1} - k_{1})+n_2-k_2+h.
    \end{split}
\end{equation}
Since $h\geq 1$, $n_1\geq k_1$, and $n_2-k_2+h\geq 0$, this difference is always $\geq 0$ and is equal to zero if and only if $n_1=k_1$ and $n_2-k_2+h=0$.
For fixed $0\leq n_1'\leq n_1+n_2$, $0\leq n_2'\leq n_2$ such that $n_1'+2n_2'+k=n_1+2n_2$, we have
\begin{align*}&\cB_{((\emptyset,2^{n_{2}'}1^{n_{1}'}),\alpha')}^{((1^{n_2},1^{n_1+n_2}),\alpha)}= \\ 
&=\left(\coprod_{\substack{\max\{1,k-n_1\}\leq k_2\leq\min\{k,2n_2\} \\ \max\{0,k_2-n_2\}\leq h< \frac{k_2}{2} \\ k_2-h=n_2-n_2'}}\mathcal{X}^1(k_2,h)\right)\coprod \left( \coprod_{\substack{\max\{1,k-n_1\}\leq k_2\leq\min\{k,2n_2\} \\ \max\{1,k_2-n_2\}\leq h\leq \frac{k_2}{2} \\ k_2-h=n_2-n_2'}}\mathcal{Y}(k_2,h)\right); \\
&=\left(\coprod_{\substack{ \max\{0,n_2-n_1'-n_2'\}\leq h \\ 1-n_2+n_2'\leq h \\ h\leq\min\{k-(n_2-n_2'),n_2-n_2'-1\}}}\mathcal{X}^1(h+n_2-n_2',h)\right)\coprod \left( \coprod_{\substack{\max\{1,n_2-n_1'-n_2'\}\leq h \\ h\leq\min\{k-(n_2-n_2'),n_2-n_2'\}}}\mathcal{Y}(h+n_2-n_2',h)\right);
\end{align*}
where the inequalities follow from the same reasoning as Lemma \ref{lem:ineq}, plus the observations that $k_2>0$, in $\mathcal{X}^1(k_2,h)$ we have $k_2-2h>0$, and in $\mathcal{Y}(k_2,h)$ we have $h>0$.
\begin{lemma}\label{lem:XYclosure}
If $\cX^1(k_2,h)\neq\emptyset\neq\cX^1(k_2+1,h+1)$, then $\cX^1(k_2+1,h+1)\subset\overline{\cX^1(k_2,h)}$; if $\cY(k_2,h)\neq\emptyset\neq\cY(k_2+1,h+1)$, then $\cY(k_2+1,h+1)\subset\overline{\cY(k_2,h)}$; if $\cX^1(k_2,h)\neq\emptyset\neq\cY(k_2+1,h+1)$, then $\cY(k_2+1,h+1)\subset\overline{\cX^1(k_2,h)}$; if $\cX^1(k_2,h)\neq\emptyset\neq\cY(k_2,h)$, then $\cX^1(k_2,h)\subset\overline{\cY(k_2,h)}$.
\end{lemma}
\begin{proof}
This is entirely analogous to the proof of Lemma \ref{lem:Xinclosure}, by choosing for each case an appropriate basis and a vector to be deformed, details are omitted.
\end{proof}
Lemma \ref{lem:XYclosure} then implies, like it happened in Section \ref{sec:mu-empty}, that there is one irreducible piece (either $\cX^1(k_2,h)$ or $\cY(k_2,h)$) of maximal dimension that contains all the other ones in its closure as positive codimension subvarieties, hence $\cB_{((\emptyset,2^{n_2'}1^{n_{1}'}),\alpha')}^{((1^{n_2},1^{n_1+n_2}),\alpha)}$ is irreducible. Notice that by \eqref{eq:diff-dd'-dimX} and \eqref{eq:diff-dd'-dimY} we have that if $n_2'=n_2-k_2+h>0$, then even the piece of maximal dimension will be of dimension strictly less than $d_{(\mu,\nu)}^{\alpha} - d_{(\mu',\nu')}^{\alpha'}$, and the shapes are not nested $(\emptyset,2^{n_{2}'}1^{n_{1}'})\not\subset (1^{n_2},1^{n_1+n_2})$. When $n_2'=0$, we have that the piece of maximal dimension is $\cX^1(n_2,0)$ if and only if $h=0$ is the minimal $h$, if and only if 
$n_1'\geq n_2$, and otherwise it is $\cY(2n_2-n_1',n_2-n_1')$ if and only if $n_2>n_1'$. If $n_1'\geq n_2$, then $\dim\cB_{((\emptyset,1^{n_{1}'}),\alpha')}^{((1^{n_2},1^{n_1+n_2}),\alpha)}=\dim\cX(n_2,0)=d_{(\mu,\nu)}^{\alpha} - d_{(\mu',\nu')}^{\alpha'}$. Otherwise, with $n_2'=0$ and $n_2>n_1'$ we have that the minimal $h$, is $h=n_2-n'_1$, which gives 
\begin{align*}
h&= n_2-n_1' \\
h&= n_2-(n_1-k_1)-(k_2-2h) \\
h&= n_2-n_1+k_1-k_2+2h \\
0&= n_2-k_2+h-n_1+k_1\\
0&=-n_1+k_1\\
n_1&=k_1,
\end{align*}
so that indeed $\dim\cB_{((\emptyset,1^{n_{1}'}),\alpha')}^{((1^{n_2},1^{n_1+n_2}),\alpha)}=\dim\cY(2n_2-n_1',n_2-n_1')=d_{(\mu,\nu)}^{\alpha} - d_{(\mu',\nu')}^{\alpha'}$.
\begin{example}Here $n_1=2$, $n_2=3$, $k=4$. In the first example $n'_1=2$, $n_2'=1>0$, so the diagrams are not nested, in the second example $n'_1=4$, $n'_2=0$ so the diagrams are nested and  the boxes are removed from different rows of the two tableaux. 
$$(\emptyset,2^{n'_2}1^{n'_1})=\left(\emptyset,\young(~~,~,~)\right)\not\subset (1^{n_2},1^{n_1+n_2})=\left(\young(~,~,~),\young(~,~,~,~,~)\right);$$
$$(\emptyset,2^{n'_2}1^{n'_1})=\left(\emptyset,\young(~,~,~,~)\right)\subset (1^{n_2},1^{n_1+n_2})=\left(\young(~,~,~),\young(~,~,~,~,~)\right),\quad \left(\young(\xx,\xx,\xx),\young(~,~,~,~,\xx)\right).$$
\end{example}

\subsubsection{}Suppose that $(\mu',\nu')=(1^{n_2'},1^{n_2'+n_1'})=(1^{n_2-k_2+h},1^{n_1+n_2-k_1-h})$, with $n_2'=n_2-k_2+h>0$. Then by Section \ref{subsub:v-in-Im}, we have $v\not\in F_k$, and $(F_k+\C v)\cap x(F_k^\perp)\supsetneq F_k\cap x(F_k^\perp)$, or equivalently that $F_k\cap x(F_k^\perp)\subset (\C v)^{\pperp}$, by Lemma \ref{lem:vImvnotinFk}. In this case we have that $k_2<2n_2$.
We define the variety
\begin{multline*}\mathcal{X}^2(k_2,h)=\{F_k\in\Gr_k^\perp(\ker(x))~|~\dim(F_k\cap \Im(x))=k_2,\dim(F_k\cap x(F_k^\perp))=k_2-2h, \\ ~v\not\in F_k,~F_k\cap x(F_k^\perp)\subset (\C v)^{\pperp} \}.\end{multline*}
Notice that $(\mathbb{C}v)^{\pperp}$ is a $2n_2-1$-dimensional space, and the restriction of $\llangle, \rrangle$ to it has rank $2n_2-2$ ($r=n_2-1$). It follows that $F_k\cap x(F_k^{\perp})\in Z:= \{S\in\Gr_{k_2-2h}^{\pperp}((\C v)^{\pperp})~|~S\cap \C v=0\}$, which is nonempty because $n_2-k_2+h>0$ implies that $k_2-2h<2n_2-1$, so it is an open dense subvariety of  $\Gr_{k_2-2h}^{\pperp}((\C v)^{\pperp})$, hence irreducible of dimension
\begin{equation*} 
\dim\Gr_{k_2-2h,2n_2-1}^{\perp n_2-1}=(k_{2}-2h)(2n_{2}-1-k_2+2h)-\tfrac{1}{2}(k_{2}-2h)(k_{2}-2h-1).
\end{equation*}
We then have
$$\mathcal{X}^2(k_2,h)=\{F_k\in (\pi_1^k)^{-1}(\pi_2^{k_2})^{-1}(Z)~|~v\not\in\pi_1^k (F_k)\}$$
Over every point $S\in Z$, we consider only the points $U\in(\pi_2^{k_2})^{-1}(S)$ such that $v\not\in U$, i.e. we are choosing $U/S\in \Gr_{2h}(S^{\pperp}/S)$, such that $U/S\cap\C v/S=0$ (where $\C v/S$ is the image of $\C v$ under the projection onto $S^{\pperp}/S$) so each fibre, if nonempty, is an open subvariety of a Grassmannian, hence irreducible of dimension
$$\dim\Gr_{2h,2n_2-2(k_2-2h)}= 2h(2n_{2} - 2k_{2}+2h).$$
This fibre is nonempty if and only if it is possible to choose $U/S$ that does not contain $\C v/S$, if and only if 
\begin{align*}
 2h&<2n_2-2(k_2-2h)\\
 2h&<2n_2-2k_2+4h\\
 0&<2n_2-2k_2+2h\\
 0&<n_2-k_2+h.
\end{align*}
Finally, over a point $U\in(\pi_2^{k_2})^{-1}\left(Z\right)$, (with $v\not\in U$) we have that
$(\pi_1^k)^{-1}(U)$ is isomorphic to an open subvariety of $\Gr^\perp_{k_1}(\ker(x)/U)$ so, as in previous sections, it is irreducible of dimension 
$$\dim\Gr_{k_1,2n_1+2n_2-k_2}^{\perp n_1}=k_{1}(2n_{2}+2n_{1} - k_{2} - k_{1}) - \tfrac{1}{2}k_{1}(k_{1}-1).$$
In conclusion $\mathcal{X}^2(k_2,h)$ is an irreducible variety of dimension
\begin{align}\label{eq:dim-var-xnot0-vnotF-1}
    &(k_{2}-2h)(2n_{2}-1-k_2+2h)-\tfrac{1}{2}(k_{2}-2h)(k_{2}-2h-1)+2h(2n_{2} - 2k_{2}+2h)+\\
    &+k_{1}(2n_{2}+2n_{1} - k_{2} - k_{1}) - \tfrac{1}{2}k_{1}(k_{1}-1). \notag
\end{align}

For $(\mu,\nu)=(1^{n_2},1^{n_1+n_2})$, $(\mu',\nu')=(1^{n_{2}-k_{2}+h},1^{n_{1}+n_2-k_{1}-h})$ we compute the difference : 

\begin{equation}\label{eq:diff-dd'-dimX-1}
	\begin{split}
    d_{(\mu,\nu)}^{\alpha} - d_{(\mu',\nu')}^{\alpha'} -\dim\mathcal{X}^2(k_2,h)&=2N(2^{n_2}1^{n_1})+n_2+n_1- \frac{1}{2} \sum_{i\geq 1} (\alpha_{i}^{2} - \alpha_{i})-2N(2^{n_{2}-k_{2}+h}1^{n_{1}-k_{1}+k_{2}-2h})\\
    &\quad -(n_{1}+n_2-k_{1}-h)+ \frac{1}{2} \sum_{i\geq 2} (\alpha_{i}^{2} - \alpha_{i})-\eqref{eq:dim-var-xnot0-vnotF-1}\\
    & = (n_{1}+ n_{2})(n_{1} + n_{2} - 1) + n_{2}(n_{2}-1) + n_{2} + n_{1}- \tfrac{1}{2}k(k-1) \\
    &\quad -(n_{1}+ n_{2}-k_{1}-h) (n_{1}+ n_{2}-k_{1}-h - 1) \\
    &\quad - (n_{2}-k_{2}+h)(n_{2}-k_{2}+h-1) \\
    & \quad - (n_{1}+n_2-k_{1}-h)-\eqref{eq:dim-var-xnot0-vnotF-1}\\
    &= 2h(n_{1} - k_{1}).
    \end{split}
\end{equation}

Since $h\geq 0$ and $n_1\geq k_1$, this difference is always greater or equal to zero and it equals zero if and only if $h=0$ or $n_1=k_1$.

For fixed $0\leq n_1'\leq n_1+n_2$, $0< n_2'\leq n_2$ such that $n_1'+2n_2'+k=n_1+2n_2$, we have
\begin{align*}\cB_{((1^{n_{2}'},1^{n_{1}'+n_2'}),\alpha')}^{((1^{n_2},1^{n_1+n_2}),\alpha)}=\coprod_{\substack{\max\{0,k-n_1\}\leq k_2\leq\min\{k,2n_2-1\} \\ \max\{0,k_2-n_2+1\}\leq h\leq \frac{k_2}{2} \\ k_2-h=n_2-n_2'}}\mathcal{X}^2(k_2,h)=\coprod_{\substack{\max\{0,n_2-n_1'-n_2'\}\leq h \\
h\leq\min\{k-(n_2-n_2'), n_2-n_2'\}}}\mathcal{X}^2(h+n_2-n_2',h).
\end{align*}
Again, the inequalities follow from Lemma \ref{lem:ineq} and from the fact that $k_2<2n_2$ and $n_2-k_2+h>0$. As in previous sections, we have that if $\cX^2(k_2,h)\neq\emptyset\neq\cX^2(k_2+1,h+1)$, then $\cX^2(k_2+1,h+1)\subset\overline{\cX^2(k_2,h)}$, so there is one irreducible piece of maximal dimension that contains all the other ones in its closure as positive codimension subvarieties. By \eqref{eq:diff-dd'-dimX-1}, the piece of maximal dimension is $\mathcal{X}^2(h+n_2-n_2',h)$ for minimal $h$, which equals zero if and only if $n_1'+n_2'\geq n_2$. Otherwise, if $n_1'+n_2'<n_2$, then the minimal $h$ is $h=n_2-n_1'-n_2'$ which gives us 
\begin{align*}
h&= n_2-n_1'-n_2' \\
h&= n_2-(n_1-k_1)-(k_2-2h)-(n_2-k_2+h) \\
h&= n_2-n_1+k_1-k_2+2h-n_2+k_2-h \\
0&=-n_1+k_1\\
n_1&=k_1.
\end{align*}
So in either case $h(n_1-k_1)=0$ when $h$ is minimal and we have indeed that $\cB_{((1^{n_{2}'},1^{n_{1}'+n_2'}),\alpha')}^{((1^{n_2},1^{n_1+n_2}),\alpha)}$ is an irreducible variety of dimension $d_{(\mu,\nu)}^{\alpha} - d_{(\mu',\nu')}^{\alpha'}$ and the diagram of $(1^{n_{2}'},1^{n_{1}'+n_2'})$ is always obtained by removing a vertical strip from $(1^{n_{2}},1^{n_{1}+n_2})$.

\subsubsection{}Suppose that $(\mu',\nu')=(1^{n_2'+n_1'},1^{n_2'})=(1^{n_1+n_2-k_1-h},1^{n_2-k_2+h})$, with $n_1'=n_1-k_1+k_2-2h>0$. Then by Section \ref{subsub:v-in-Im}, we have $v\not\in F_k$, and $(F_k+\C v)\cap x(F_k^\perp)= F_k\cap x(F_k^\perp)$, or equivalently that $F_k\cap x(F_k^\perp)\not\subset (\C v)^{\pperp}$, by Lemma \ref{lem:vImvnotinFk}. In this case we have that $k_2<2n_2$.

We define the variety
\begin{multline*}\mathcal{X}^3(k_2,h)=\{F_k\in\Gr_k^\perp(\ker(x))~|~\dim(F_k\cap \Im(x))=k_2,\dim(F_k\cap x(F_k^\perp))=k_2-2h, \\ ~v\not\in F_k,~F_k\cap x(F_k^\perp)\not\subset (\C v)^{\pperp} \}.\end{multline*}
Notice that $F_k\cap x(F_k^{\perp})\in Z:= \{S\in\Gr^{\pperp}_{k_2-2h}(\Im(x))~|~S\not\subset (\C v)^{\pperp},~S\cap \C v=0\}$ which is impossible if $k_2-2h=0$, and otherwise, when $k_2-2h>0$, is an open dense subvariety of  $\Gr_{k_2-2h}^{\pperp}(\Im(x))$, hence irreducible of dimension
$$(k_{2} - 2h)(2n_{2} - k_{2} +2h) - \tfrac{1}{2} (k_{2}-2h) (k_{2}-2h-1).$$
We then have
$$\mathcal{X}^3(k_2,h)=\{F_k\in (\pi_1^k)^{-1}(\pi_2^{k_2})^{-1}(Z)~|~v\not\in\pi_1^k (F_k)\}=(\pi_1^k)^{-1}(\pi_2^{k_2})^{-1}(Z)$$
because, if $S\in Z$,  since $S\not\subset (\C v)^{\pperp}$, then $\C v\not\in S^{\pperp}$, hence for all $U\in(\pi_2^{k_2})^{-1}(S)=\Gr_{2h}(S^{\pperp}/S)$ automatically $\C v\not\subset U$. So we get, for each $S\in Z$, that the fibre is irreducible and
$$\dim(\pi_2^{k_2})^{-1}(S)=\dim\Gr_{2h,2n_2-2(k_2-2h)}= 2h(2n_{2} - 2k_{2}+2h).$$
Finally, over a point $U\in(\pi_2^{k_2})^{-1}\left(Z\right)$, we have that
$(\pi_1^k)^{-1}(U)$ is isomorphic to an open dense subvariety of $\Gr^\perp_{k_1}(\ker(x)/U)$ so, as in previous sections, it is irreducible of dimension 
$$\dim\Gr_{k_1,2n_1+2n_2-k_2}^{\perp n_1}=k_{1}(2n_{2}+2n_{1} - k_{2} - k_{1}) - \tfrac{1}{2}k_{1}(k_{1}-1).$$
In conclusion $\mathcal{X}^3(k_2,h)$ is an irreducible variety of dimension
\begin{align}\label{eq:dim-var-xnot0-vnotF-2}
    &(k_{2}-2h)(2n_{2}-k_2+2h)-\tfrac{1}{2}(k_{2}-2h)(k_{2}-2h-1)+2h(2n_{2} - 2k_{2}+2h)+\\
    &+k_{1}(2n_{2}+2n_{1} - k_{2} - k_{1}) - \tfrac{1}{2}k_{1}(k_{1}-1). \notag
\end{align}
For $(\mu,\nu)=(1^{n_2},1^{n_1+n_2})$, $(\mu',\nu')=(1^{n_{1}+n_2-k_{1}-h}, 1^{n_{2}-k_{2}+h})$ we compute the difference : 

\begin{equation}\label{eq:diff-dd'-dimX-2}
	\begin{split}
    d_{(\mu,\nu)}^{\alpha} - d_{(\mu',\nu')}^{\alpha'} -\dim\mathcal{X}^3(k_2,h)&=2N(2^{n_2}1^{n_1})+n_2+n_1- \frac{1}{2} \sum_{i\geq 1} (\alpha_{i}^{2} - \alpha_{i})-2N(2^{n_{2}-k_{2}+h}1^{n_{1}-k_{1}+k_{2}-2h})\\
    &\quad -(n_2-k_{2}+h)+ \frac{1}{2} \sum_{i\geq 2} (\alpha_{i}^{2} - \alpha_{i})-\eqref{eq:dim-var-xnot0-vnotF-2}\\
    & = (n_{1}+ n_{2})(n_{1} + n_{2} - 1) + n_{2}(n_{2}-1) + n_{2} + n_{1}- \tfrac{1}{2}k(k-1) \\
    &\quad -(n_{1}+ n_{2}-k_{1}-h) (n_{1}+ n_{2}-k_{1}-h - 1) \\
    &\quad - (n_{2}-k_{2}+h)(n_{2}-k_{2}+h-1) \\
    & \quad - (n_{2}-k_{2}+h)-\eqref{eq:dim-var-xnot0-vnotF-2}\\
    &= (2h+1)(n_{1} - k_{1}).
    \end{split}
\end{equation}
Since $h\geq 0$ and $n_1\geq k_1$, this difference is always greater or equal to zero and it equals zero if and only if $n_1=k_1$.

For fixed $0< n_1'\leq n_1+n_2$, $0\leq n_2'\leq n_2$ such that $n_1'+2n_2'+k=n_1+2n_2$, we have
\begin{align*}\cB_{((1^{n_{1}'+n'_2}, 1^{n_{2}'}),\alpha')}^{((1^{n_2},1^{n_1+n_2}),\alpha)}=\coprod_{\substack{\max\{0,k-n_1\}\leq k_2\leq\min\{k,2n_2-1\} \\ \max\{0,k_2-n_2\}\leq h< \frac{k_2}{2} \\ k_2-h=n_2-n_2'}}\mathcal{X}^3(k_2,h)=\coprod_{\substack{\max\{k-n_1,n_2-n_2'\}\leq k_2 \\
k_2\leq\min\{k, 2n_2-2n_2',2n_2-1\}}}\mathcal{X}^3(k_2,k_2-n_2+n_2').
\end{align*}
As in previous sections, the inequalities follow from Lemma \ref{lem:ineq} plus the stricter inequalities that we found in this section, also if $\cX^3(k_2,h)\neq\emptyset\neq\cX^3(k_2+1,h+1)$, then $\cX^3(k_2+1,h+1)\subset\overline{\cX^3(k_2,h)}$. It follows that there is one irreducible piece of maximal dimension that contains all the other ones in its closure as positive codimension subvarieties and hence $\cB_{((1^{n_{1}'+n'_2}, 1^{n_{2}'}),\alpha')}^{((1^{n_2},1^{n_1+n_2}),\alpha)}$ is an irreducible variety. By \eqref{eq:diff-dd'-dimX-2} we have that the piece of maximal dimension is $\mathcal{X}^3(k_2,k_2-n_2+n_2')$ for minimal $k_2$. 
Notice that the shape $(1^{n_{1}'+n'_2}, 1^{n_{2}'})$ is contained in the shape $(1^{n_2},1^{n_1+n_2})$ (and hence obtained by removing boxes appropriately) if and only if $n_2\geq n_1'+n_2'$,
\begin{align*}
n_2 & \geq n_1'+n_2' \\
 n_2-n_1'-n_2'&\geq 0 \\
 2n_2-n_2+n_1-n_1-n_1'-2n_2'+n_2'-k+k&\geq 0\\
 n-n_2-n_1-n+n_2'+k&\geq 0\\
 k-n_1&\geq n_2-n_2'
\end{align*}
if and only if $\max\{k-n_1,n_2-n_2'\}=k-n_1$, if and only if the minimal $k_2$ is $k_2=k-n_1$, if and only if 
\begin{align*}
k_2& =k-n_1 \\
    k_2&=k_1+k_2-n_1 \\
    0 &= k_1-n_1 \\
    n_1&=k_1
\end{align*}
and as a conclusion $\dim\cB_{((1^{n_{1}'+n'_2}, 1^{n_{2}'}),\alpha')}^{((1^{n_2},1^{n_1+n_2}),\alpha)}=d_{(\mu,\nu)}^{\alpha} - d_{(\mu',\nu')}^{\alpha'}$ if and only if the shape is obtained by removing boxes as desired, and is strictly lower otherwise.
\begin{example}Here $n_1=2$, $n_2=3$. In the first example $k=2$, $n'_1+n'_2=2+2>3=n_2$, so the diagrams are not nested, in the second example $k=5$, $n'_1+n'_2=1+1=2\leq n_2$, so the diagrams are nested and  the boxes are removed from different rows of the two tableaux. 
$$(1^{n'_1+n'_2},1^{n'_1})=\left(\young(~,~,~,~),\young(~,~)\right)\not\subset (1^{n_2},1^{n_1+n_2})=\left(\young(~,~,~),\young(~,~,~,~,~)\right);$$
$$(1^{n'_1+n'_2},1^{n'_1})=\left(\young(~,~),\young(~)\right)\subset (1^{n_2},1^{n_1+n_2})=\left(\young(~,~,~),\young(~,~,~,~,~)\right),\quad \left(\young(~,~,\xx),\young(~,\xx,\xx,\xx,\xx)\right).$$
\end{example}

\subsection{\underline{$\mu= 1^{n_1+n_2}, \nu=1^{n_2}$}}
Here $n_1,n_2>0$. In this case we have $x\neq 0$, $x^2=0$ and $v\in\ker(x)\setminus \Im(x)$, so $\ker(x)\cap(\C[x]v)^\perp=\ker(x)\cap(\C v)^\perp$ is a $2n_1+2n_2-1$-dimensional space where the restriction of $\langle,\rangle$ has rank $2n_1-2$ ($r=n_1-1$). We define $k_1$, $k_2$ and $h$ as in Section \ref{section:xnonzero}. According to Section \ref{subsub:v-in-ker}, there are three possibilities for $(\mu',\nu')$, which we now examine separately.
\subsubsection{}Suppose that $(\mu',\nu')=(\emptyset,2^{n_2'}1^{n_1'})=(\emptyset,2^{n_2-k_2+h}1^{n_1-k_1+k_2-2h})$. Then by Section \ref{subsub:v-in-Im}, we have $v\in F_k\setminus F_k\cap\Im(x)$, which implies that $k_1\geq 1$. We define the following variety
\begin{multline*}
\mathcal{X}^4(k_2,h)=\{F_k\in\Gr_k^\perp(\ker(x)\cap(\C[x]v)^\perp)~|~\dim(F_k\cap \Im(x))=k_2,\\ \dim(F_k\cap x(F_k^\perp))=k_2-2h,v\in F_k\setminus F_k\cap\Im(x) \} 
\end{multline*}
Then
$$\mathcal{X}^4(k_2,h)=\{F_k\in(\pi_1^k)^{-1}(\pi_2^{k_2})^{-1}\left(\Gr_{k_2-2h}^{\pperp}(\Im(x))\right)~|~v\in F_k\}.$$
The base of the iterated fibre bundle is the isotropic Grassmannian $\Gr_{k_2-2h}^{\pperp}(\Im(x))$ which is irreducible of dimension
$$(k_2-2h)(2n_2-k_2+2h) - \tfrac{1}{2}(k_2-2h)(k_2-2h-1).$$
Over each point $S\in \Gr_{k_2-2h}^{\pperp}(\Im(x))$, we have $(\pi_2^{k_2})^{-1}(S)\simeq \Gr_{2h}(S^{\pperp}/S)$, so each fibre is irreducible of dimension
$$ 2h(2n_{2} - 2k_{2}+2h).$$
Finally, for each $U\in(\pi_2^{k_2})^{-1}\left(\Gr_{k_2-2h}^{\pperp}(\Im(x))\right)$, we only consider the points $F\in (\pi_1^k)^{-1}(U)$ such that $v\in F$, so we are choosing 

$$F/(U+\C v)\in\Gr_{k_1-1}^\perp(\ker(x)\cap(\C[x]v)^\perp/(U+\C v))\simeq \Gr_{k_1-1,2n_1  + 2n_2 - k_2 - 2}^{\perp,n_1-1}$$
which is irreducible of dimension $(k_1-1)(2n_2+2n_1-k_2-k_1-1)-\tfrac{1}{2}(k_{1}-1)(k_{1}-2)$.

In conclusion, $\mathcal{X}^4(k_2,h)$ is an irreducible variety of dimension 
\begin{align}\label{eq:dim-var-vnotIm-vinF}
    & (k_2-2h)(2n_2-k_2+2h) - \tfrac{1}{2}(k_2-2h)(k_2-2h-1)+2h(2n_{2} - 2k_{2}+2h)+\\
    &~+ (k_1-1)(2n_2+2n_1-k_2-k_1-1)-\tfrac{1}{2}(k_{1}-1)(k_{1}-2). \notag
\end{align}
For $(\mu,\nu)=(1^{n_1+n_2},1^{n_2})$, $(\mu',\nu')=(\emptyset,2^{n_2-k_2+h}1^{n_1-k_1+k_{2}-2h})$ we compute the difference : 

\begin{equation}\label{eq:diff-dd'-dimX-3}
	\begin{split}
    d_{(\mu,\nu)}^{\alpha} - d_{(\mu',\nu')}^{\alpha'} -\dim\mathcal{X}^4(k_2,h)&=2N(2^{n_2}1^{n_1})+n_2- \frac{1}{2} \sum_{i\geq 1} (\alpha_{i}^{2} - \alpha_{i})-2N(2^{n_{2}-k_{2}+h}1^{n_{1}-k_{1}+k_{2}-2h})\\
    &\quad -2(n_2-k_{2}+h)-(n_1-k_1+k_{2}-2h)+ \frac{1}{2} \sum_{i\geq 2} (\alpha_{i}^{2} - \alpha_{i})-\eqref{eq:dim-var-vnotIm-vinF}\\
    & = (n_{1}+ n_{2})(n_{1} + n_{2} - 1) + n_{2}(n_{2}-1) + n_{2}- \tfrac{1}{2}k(k-1) \\
    &\quad -(n_{1}+ n_{2}-k_{1}-h) (n_{1}+ n_{2}-k_{1}-h - 1) \\
    &\quad - (n_{2}-k_{2}+h)(n_{2}-k_{2}+h-1) \\
    & \quad -2(n_2-k_{2}+h)-(n_1-k_1+k_{2}-2h)-\eqref{eq:dim-var-vnotIm-vinF}\\
    &= (2h+1)(n_{1} - k_{1})+(n_2-k_2+h).
    \end{split}
\end{equation}
which is always greater or equal to zero and equals zero if and only if $n_1=k_1$ and $n_2-k_2+h=0$.

For fixed $0\leq n_1'\leq n_1+n_2$, $0\leq n_2'\leq n_2$ such that $n_1'+2n_2'+k=n_1+2n_2$, we have
\begin{align*}\cB_{((\emptyset,2^{n_{2}'}1^{n_1'}),\alpha')}^{((1^{n_1+n_2},1^{n_2}),\alpha)}=\coprod_{\substack{\max\{0,k-n_1\}\leq k_2\leq\min\{k-1,2n_2\} \\ \max\{0,k_2-n_2\}\leq h\leq \frac{k_2}{2} \\ k_2-h=n_2-n_2'}}\mathcal{X}^4(k_2,h)=\coprod_{\substack{\max\{k-n_1,n_2-n_2'\}\leq k_2 \\
k_2\leq\min\{k-1, 2n_2-2n_2'\}}}\mathcal{X}^4(k_2,k_2-n_2+n_2')
\end{align*}
As before, if $\cX^4(k_2,h)\neq\emptyset\neq\cX^4(k_2+1,h+1)$, then $\cX^4(k_2+1,h+1)\subset\overline{\cX^4(k_2,h)}$. It follows that there is one maximal dimensional piece that contains the others in its closure as positive codimension subvarieties, hence $\cB_{((\emptyset,2^{n_{2}'}1^{n_1'}),\alpha')}^{((1^{n_1+n_2},1^{n_2}),\alpha)}$ is irreducible, and its dimension is the dimension of $\mathcal{X}^4(k_2,k_2-n_2+n_2')$ for minimal $k_2$. If $n_2'=n_2-k_2+h>0$, then all pieces are strictly lower dimensional than $d_{(\mu,\nu)}^{\alpha} - d_{(\mu',\nu')}^{\alpha'}$ and the shape of $(\mu',\nu')$ is not contained in the shape of $(\mu,\nu)$. Otherwise, when $n_2'=0$, then $n=n_1'+k$ so we have that the minimal $k_2$ is $k-n_1$ if and only if $k-n_1\geq n_2$,
\begin{align*}
 k-n_1\geq n_2 \\
 k\geq n_2+n_1 \\
 n-n_1'\geq 2n_2+n_1-n_2 \\
 n-n_1'\geq n-n_2 \\
 n_2\geq n_1'
\end{align*}
if and only if the shape $(\mu',\nu')$ is contained in the shape of $(\mu,\nu)$. Also notice that $k_2=k-n_1$, is equivalent to $n_1=k_1$ so that indeed when $n_1'\leq n_2$ we have $\dim\cB_{((\emptyset,2^{n_{2}'}1^{n_1'}),\alpha')}^{((1^{n_1+n_2},1^{n_2}),\alpha)}=\dim\mathcal{X}^4(k-n_1,k-n_1-n_2)=d_{(\mu,\nu)}^{\alpha} - d_{(\mu',\nu')}^{\alpha'}$, and is strictly lower otherwise.
\begin{example}Here $n_1=2$, $n_2=3$. In the first example $k=4$, $n'_1=2$, $n_2'=1>0$, in the second example $k=4$ and, while $n'_2=0$, $n'_1=4>3=n_2$, so in both cases the diagrams are not nested. In the third example $k=6$, $n'_2=0$ and $n'_1=2\leq 3=n_2$, so the diagrams are nested and  the boxes are removed from different rows of the two tableaux. 
$$(\emptyset,2^{n'_2}1^{n'_1})=\left(\emptyset,\young(~~,~,~)\right)\not\subset (1^{n_1+n_2},1^{n_2})=\left(\young(~,~,~,~,~),\young(~,~,~)\right);$$ $$(\emptyset,2^{n'_2}1^{n'_1})=\left(\emptyset,\young(~,~,~,~)\right)\not\subset (1^{n_1+n_2},1^{n_2})=\left(\young(~,~,~,~,~),\young(~,~,~)\right);$$
$$(\emptyset,2^{n'_2}1^{n'_1})=\left(\emptyset,\young(~,~)\right)\subset (1^{n_1+n_2},1^{n_2})=\left(\young(~,~,~,~,~),\young(~,~,~)\right),\quad \left(\young(\xx,\xx,\xx,\xx,\xx),\young(~,~,\xx)\right).$$
\end{example}

\subsubsection{}Suppose that $(\mu',\nu')=(1^{n_2'},1^{n_2'+n_1'})=(1^{n_2-k_2+h},1^{n_1+n_2-k_1-h})$, with $n_2'=n_2-k_2+h>0$. Then by Section \ref{subsub:v-in-Im}, we have $v\not\in F_k$, and $(F_k+\C v)\cap x(F_k^\perp)\supsetneq F_k\cap x(F_k^\perp)$, or equivalently that $(v+F_k)\cap x(F_k^\perp)\neq \emptyset$, by Lemma \ref{lem:vnotImvnotinFk}. In this case we have that $k_1\geq 1$ (otherwise if $F_k\subset \Im(x)$, then $(v+F_k)\cap \Im(x)=\emptyset$). Notice that since $v\not\in\Im(x)$, the condition that $v\not\in F_k,~(v+F_k)\cap x(F_k^\perp)\neq\emptyset$ is only possible if $x(F_k^\perp)\not\subset F_k$, equivalently,
\begin{align*}x(F_k^\perp)\setminus F_k\neq \emptyset \iff
x(F_k^\perp)\setminus F_k\cap x(F_k^\perp) \neq \emptyset \iff \dim(x(F_k^\perp))>\dim(F_k\cap x(F_k^\perp)).
\end{align*} 
Since $x(F_k^\perp)=(F_k\cap\Im(x))^{\pperp}$ by Lemma \ref{lem:F-rad-j}, the dimension condition can be expressed as
$$2n_2-k_2>k_2-2h \iff 2n_2-2k_2+2h>0 \iff n_2-k_2+h>0$$
which has to be satisfied.
We define the variety
\begin{multline*}\mathcal{X}^5(k_2,h)=\{F_k\in\Gr_k^\perp(\ker(x)\cap(\C[x]v)^\perp)~|~\dim(F_k\cap \Im(x))=k_2, \\ \dim(F_k\cap x(F_k^\perp))=k_2-2h,~v\not\in F_k,~(v+F_k)\cap x(F_k^\perp)\neq \emptyset \}\end{multline*}
which we can write as
$$\mathcal{X}^5(k_2,h)=\{F_k\in(\pi_1^k)^{-1}(\pi_2^{k_2})^{-1}\left(\Gr_{k_2-2h}^{\pperp}(\Im(x))\right)~|~v\not\in F_k,(v+F_k)\cap (\pi_1^k(F_k))^{\pperp}\neq \emptyset\}.$$
The base of the iterated fibre bundle is the isotropic Grassmannian $\Gr_{k_2-2h}^{\pperp}(\Im(x))$ which is irreducible of dimension
$$(k_2-2h)(2n_2-k_2+2h) - \tfrac{1}{2}(k_2-2h)(k_2-2h-1).$$
Over each point $S\in \Gr_{k_2-2h}^{\pperp}(\Im(x))$, we have $(\pi_2^{k_2})^{-1}(S)\simeq \Gr_{2h}(S^{\pperp}/S)$, so each fibre is irreducible of dimension
$$ 2h(2n_{2} - 2k_{2}+2h).$$
Finally, for each $U\in(\pi_2^{k_2})^{-1}\left(\Gr_{k_2-2h}^{\pperp}(\Im(x))\right)$, we only consider the points $F\in (\pi_1^k)^{-1}(U)$ such that $v\not\in F$, $(v+F)\cap (U)^{\pperp}\neq \emptyset$. 

Suppose $L=(v+F)\cap (U)^{\pperp}\neq\emptyset$ and $w\in L$, then $w\in(U)^{\pperp}\setminus F\cap U^{\pperp}=(U)^{\pperp}\setminus U\cap U^{\pperp}$, viceversa given $w\in(U)^{\pperp}\setminus U\cap U^{\pperp}$, we have that $L\neq\emptyset$ if and only if $w=v+f$ with $f\in F$, if and only if $w-v\in F$. $L$ (if nonempty) is an affine space of dimension equal to $\dim F\cap U^{\pperp}=\dim U\cap U^{\pperp}$ and all points in $L$ are of the form $w+u$, with $w\in L$ and $u\in U\cap U^{\pperp}$. It follows that if $w,w'\in(U)^{\pperp}\setminus U\cap U^{\pperp}$, then $(U+\C(w-v))/U=(U+\C(w'-v))/U$ if and only if $w-w'\in U\cap U^{\pperp}$. 

So, we have an isomorphism of varieties
\begin{multline*}
\mathcal{D}:=\{(w,D)~|~w\in((U)^{\pperp}\setminus U\cap U^{\pperp})/(U\cap U^{\pperp}),D\in\Gr_{k_1-1}^\perp(\ker(x)\cap(\C[x]v)^\perp/(U+\C (w-v)),v\not\in D\}\\
\simeq \{F\in (\pi_1^k)^{-1}(U)~|~v\not\in F,~(v+F)\cap (U)^{\pperp}\neq \emptyset\}\end{multline*}
given by $(w,D)\mapsto \tilde{D}$, where $\tilde{D}$ is any pre-image of $D$ under the projection $\ker(x)\cap(\C[x]v)^\perp/U\to\ker(x)\cap(\C[x]v)^\perp/(U+\C (w-v)).$

The variety $\mathcal{D}$ is by definition a fibre bundle with the map $(w,D)\mapsto w$, with base $((U)^{\pperp}\setminus U\cap U^{\pperp})/U\cap U^{\pperp}=((U)^{\pperp}\setminus \pi_2^{k_2}(U))/\pi_2^{k_2}(U)$, which is an irreducible variety of dimension $\dim((U)^{\pperp})-\dim(U\cap U^{\pperp})=2n_2-k_2-(k_2-2h)=2n_2-2k_2+2h$ (isomorphic to a vector space minus the origin), and fibre isomorphic to an open dense subvariety (to guarantee that $v\not\in D$) of 
$$\Gr_{k_1-1}^\perp(\ker(x)\cap(\C[x]v)^\perp/(U+\C (w-v))\simeq \Gr_{k_1-1,2n_1  + 2n_2 - k_2 - 2}^{\perp,n_1-1}$$
which is irreducible of dimension $(k_1-1)(2n_2+2n_1-k_2-k_1-1)-\tfrac{1}{2}(k_{1}-1)(k_{1}-2)$.
In conclusion, $\mathcal{X}^5(k_2,h)$ is an irreducible variety of dimension 
\begin{align}\label{eq:dim-var-vnotIm-vnotF-1}
    & (k_2-2h)(2n_2-k_2+2h) - \tfrac{1}{2}(k_2-2h)(k_2-2h-1)+2h(2n_{2} - 2k_{2}+2h)+\\
    &~+2n_2-2k_2+2h+ (k_1-1)(2n_2+2n_1-k_2-k_1-1)-\tfrac{1}{2}(k_{1}-1)(k_{1}-2). \notag
\end{align}
For $(\mu,\nu)=(1^{n_1+n_2},1^{n_2})$, $(\mu',\nu')=(1^{n_2-k_2+h},1^{n_1+n_2-k_1-h})$ we compute the difference:
\begin{equation}\label{eq:diff-dd'-dimX-4}
	\begin{split}
    d_{(\mu,\nu)}^{\alpha} - d_{(\mu',\nu')}^{\alpha'} -\dim\mathcal{X}^5(k_2,h)&=2N(2^{n_2}1^{n_1})+n_2- \frac{1}{2} \sum_{i\geq 1} (\alpha_{i}^{2} - \alpha_{i})-2N(2^{n_{2}-k_{2}+h}1^{n_{1}-k_{1}+k_{2}-2h})\\
    &\quad -(n_1+n_2-k_1-h)+ \frac{1}{2} \sum_{i\geq 2} (\alpha_{i}^{2} - \alpha_{i})-\eqref{eq:dim-var-vnotIm-vnotF-1}\\
    & = (n_{1}+ n_{2})(n_{1} + n_{2} - 1) + n_{2}(n_{2}-1) + n_{2}- \tfrac{1}{2}k(k-1) \\
    &\quad -(n_{1}+ n_{2}-k_{1}-h) (n_{1}+ n_{2}-k_{1}-h - 1) \\
    &\quad - (n_{2}-k_{2}+h)(n_{2}-k_{2}+h-1) \\
    & \quad -(n_1+n_2-k_1-h)-\eqref{eq:dim-var-vnotIm-vnotF-1}\\
    &= (2h+1)(n_{1} - k_{1}).
    \end{split}
\end{equation}
which is always greater or equal to zero and equals zero if and only if $n_1=k_1$.
For fixed $0\leq n_1'\leq n_1+n_2$, $0\leq n_2'\leq n_2$ such that $n_1'+2n_2'+k=n_1+2n_2$, we have
\begin{align*}\cB_{((1^{n_{2}'},1^{n_1'+n_2'}),\alpha')}^{((1^{n_1+n_2},1^{n_2}),\alpha)}=\coprod_{\substack{\max\{0,k-n_1\}\leq k_2\leq\min\{k-1,2n_2\} \\ \max\{0,k_2-n_2\}\leq h\leq \frac{k_2}{2} \\ k_2-h=n_2-n_2'}}\mathcal{X}^5(k_2,h)=\coprod_{\substack{\max\{k-n_1,n_2-n_2'\}\leq k_2 \\
k_2\leq\min\{k-1, 2n_2-2n_2'\}}}\mathcal{X}^5(k_2,k_2-n_2+n_2')
\end{align*}
As before, if $\cX^5(k_2,h)\neq\emptyset\neq\cX^5(k_2+1,h+1)$, then $\cX^5(k_2+1,h+1)\subset\overline{\cX^5(k_2,h)}$. It follows that there is one irreducible piece of maximal dimension that contains all the other ones in its closure as positive codimension subvarieties, hence $\cB_{((1^{n_{2}'},1^{n_{1}'+n'_2}),\alpha')}^{((1^{n_1+n_2},1^{n_2}),\alpha)}$ is an irreducible variety. By \eqref{eq:diff-dd'-dimX-4} we have that the piece of maximal dimension is $\mathcal{X}^5(k_2,k_2-n_2+n_2')$ for minimal $k_2$. 
Notice that the shape $(1^{n_{2}'},1^{n_{1}'+n'_2})$ is contained in the shape $(1^{n_1+n_2},1^{n_2})$ (and hence obtained by removing boxes appropriately) if and only if $n_2-n_1'-n_2'\geq 0$
\begin{align*}
 n_2-n_1'-n_2'&\geq 0 \\
 2n_2-n_2+n_1-n_1-n_1'-2n_2'+n_2'-k+k&\geq 0\\
 n-n_2-n_1-n+n_2'+k&\geq 0\\
 k-n_1&\geq n_2-n_2'
\end{align*}
if and only if $\max\{k-n_1,n_2-n_2'\}=k-n_1$, if and only if the minimum possible $k_2$ is $k_2=k-n_1$, if and only if 
\begin{align*}
k_2& =k-n_1 \\
    k_2&=k_1+k_2-n_1 \\
    0 &= k_1-n_1 \\
    n_1&=k_1
\end{align*}
and as a conclusion $\dim\cB_{((1^{n_{1}'+n'_2}, 1^{n_{2}'}),\alpha')}^{((1^{n_2},1^{n_1+n_2}),\alpha)}=d_{(\mu,\nu)}^{\alpha} - d_{(\mu',\nu')}^{\alpha'}$ if and only if the shape is obtained by removing boxes as desired and is strictly lower otherwise.
\begin{example}Here $n_1=2$, $n_2=3$. In the first example $k=2$, $n'_1+n'_2=2+2>3=n_2$, so the diagrams are not nested, in the second example $k=5$, $n'_1+n'_2=1+1=2\leq n_2$, so the diagrams are nested and  the boxes are removed from different rows of the two tableaux. 
$$(1^{n'_1},1^{n'_1+n'_2})=\left(\young(~,~),\young(~,~,~,~)\right)\not\subset (1^{n_1+n_2},1^{n_2})=\left(\young(~,~,~,~,~),\young(~,~,~)\right);$$
$$(1^{n'_1},1^{n'_1+n'_2})=\left(\young(~),\young(~,~)\right)\subset (1^{n_1+n_2},1^{n_2})=\left(\young(~,~,~,~,~),\young(~,~,~)\right),\quad \left(\young(~,\xx,\xx,\xx,\xx),\young(~,~,\xx)\right).$$
\end{example}
\subsubsection{}Suppose that $(\mu',\nu')=(1^{n_2'+n_1'},1^{n_2'})=(1^{n_1+n_2-k_1-h},1^{n_2-k_2+h})$, with $n_1'=n_1-k_1+k_2-2h>0$. Then by Section \ref{subsub:v-in-Im}, we have $v\not\in F_k$, and $(F_k+\C v)\cap x(F_k^\perp)= F_k\cap x(F_k^\perp)$, or equivalently that $(v+F_k)\cap x(F_k^\perp)= \emptyset$, by Lemma \ref{lem:vnotImvnotinFk}.
We define the variety
\begin{multline*}\mathcal{X}^6(k_2,h)=\{F_k\in\Gr_k^\perp(\ker(x)\cap(\C[x]v)^\perp)~|~\dim(F_k\cap \Im(x))=k_2, \\ \dim(F_k\cap x(F_k^\perp))=k_2-2h,~(v+F_k)\cap x(F_k^\perp)=\emptyset \}\end{multline*}
which we can write as
$$\mathcal{X}^6(k_2,h)=\{F_k\in(\pi_1^k)^{-1}(\pi_2^{k_2})^{-1}\left(\Gr_{k_2-2h}^{\pperp}(\Im(x))\right)~|~(v+F_k)\cap (\pi_1^k(F_k))^{\pperp}= \emptyset\}.$$
The base of the iterated fibre bundle is the isotropic Grassmannian $\Gr_{k_2-2h}^{\pperp}(\Im(x))$ which is irreducible of dimension
$$(k_2-2h)(2n_2-k_2+2h) - \tfrac{1}{2}(k_2-2h)(k_2-2h-1).$$
Over each point $S\in \Gr_{k_2-2h}^{\pperp}(\Im(x))$, we have $(\pi_2^{k_2})^{-1}(S)\simeq \Gr_{2h}(S^{\pperp}/S)$, so each fibre is irreducible of dimension
$$ 2h(2n_{2} - 2k_{2}+2h).$$
Finally, for each $U\in(\pi_2^{k_2})^{-1}\left(\Gr_{k_2-2h}^{\pperp}(\Im(x))\right)$, we only consider the points $F\in (\pi_1^k)^{-1}(U)$ such that $(v+F)\cap (U)^{\pperp}= \emptyset$, which is an open dense subvariety of 
$$\Gr_{k_1}^\perp(\ker(x)\cap(\C[x]v)^\perp/U)\simeq \Gr_{k_1,2n_1+2n_2-1-k_2}^{\perp,n_1-1}$$
which is irreducible of dimension $k_1(2n_2+2n_1-k_2-k_1-1)-\tfrac{1}{2}k_{1}(k_{1}-1)$.

In conclusion, $\mathcal{X}^6(k_2,h)$ is an irreducible variety of dimension 
\begin{align}\label{eq:dim-var-vnotIm-vnotF-2}
    & (k_2-2h)(2n_2-k_2+2h) - \tfrac{1}{2}(k_2-2h)(k_2-2h-1)+2h(2n_{2} - 2k_{2}+2h)+\\
    &~k_1(2n_2+2n_1-k_2-k_1-1)-\tfrac{1}{2}k_{1}(k_{1}-1). \notag
\end{align}
For $(\mu,\nu)=(1^{n_1+n_2},1^{n_2})$, $(\mu',\nu')=(1^{n_1+n_2-k_1-h},1^{n_2-k_2+h})$ we compute the difference:
\begin{equation}\label{eq:diff-dd'-dimX-5}
	\begin{split}
    d_{(\mu,\nu)}^{\alpha} - d_{(\mu',\nu')}^{\alpha'} -\dim\mathcal{X}^6(k_2,h)&=2N(2^{n_2}1^{n_1})+n_2- \frac{1}{2} \sum_{i\geq 1} (\alpha_{i}^{2} - \alpha_{i})-2N(2^{n_{2}-k_{2}+h}1^{n_{1}-k_{1}+k_{2}-2h})\\
    &\quad -(n_2-k_2+h)+ \frac{1}{2} \sum_{i\geq 2} (\alpha_{i}^{2} - \alpha_{i})-\eqref{eq:dim-var-vnotIm-vnotF-2}\\
    & = (n_{1}+ n_{2})(n_{1} + n_{2} - 1) + n_{2}(n_{2}-1) + n_{2}- \tfrac{1}{2}k(k-1) \\
    &\quad -(n_{1}+ n_{2}-k_{1}-h) (n_{1}+ n_{2}-k_{1}-h - 1) \\
    &\quad - (n_{2}-k_{2}+h)(n_{2}-k_{2}+h-1) \\
    & \quad -(n_2-k_2+h)-\eqref{eq:dim-var-vnotIm-vnotF-2}\\
    &= 2h(n_{1} - k_{1}).
    \end{split}
\end{equation}
Since $h\geq 0$ and $n_1\geq k_1$, this difference is always greater or equal to zero and it equals zero if and only if $h=0$ or $n_1=k_1$.

For fixed $0< n_1'\leq n_1+n_2$, $0\leq n_2'\leq n_2$ such that $n_1'+2n_2'+k=n_1+2n_2$, we have
\begin{align*}\cB_{((1^{n_{1}'+n_2'},1^{n_{2}'}),\alpha')}^{((1^{n_1+n_2},1^{n_2}),\alpha)}=\coprod_{\substack{\max\{0,k-n_1\}\leq k_2\leq\min\{k,2n_2\} \\ \max\{0,k_2-n_2\}\leq h\leq \frac{k_2}{2} \\ k_2-h=n_2-n_2'}}\mathcal{X}^6(k_2,h)=\coprod_{\substack{\max\{0,n_2-n_1'-n_2'\}\leq h \\
h\leq\min\{k-(n_2-n_2'), n_2-n_2'\}}}\mathcal{X}^6(h+n_2-n_2',h)
\end{align*}
As in previous cases, if $\cX^6(k_2,h)\neq\emptyset\neq\cX^6(k_2+1,h+1)$, then $\cX^6(k_2+1,h+1)\subset\overline{\cX^6(k_2,h)}$. It follows that there is one irreducible piece of maximal dimension that contains all the other ones in its closure as positive codimension subvarieties, hence $\cB_{((1^{n_{2}'},1^{n_{1}'+n'_2}),\alpha')}^{((1^{n_1+n_2},1^{n_2}),\alpha)}$ is an irreducible variety. By \eqref{eq:diff-dd'-dimX-5}, the piece of maximal dimension is $\mathcal{X}^6(h+n_2-n_2',h)$ for minimal $h$. This minimal $h$ equals zero if and only if $n_1'+n_2'\geq n_2$. Otherwise, if $n_1'+n_2'<n_2$, then the minimal $h$ is $h=n_2-n_1'-n_2'$ which gives us 
\begin{align*}
h&= n_2-n_1'-n_2' \\
h&= n_2-(n_1-k_1)-(k_2-2h)-(n_2-k_2+h) \\
h&= n_2-n_1+k_1-k_2+2h-n_2+k_2-h \\
0&=-n_1+k_1\\
n_1&=k_1.
\end{align*}
So we always have that $\cB_{((1^{n_{1}'+n_2'},1^{n_{2}'}),\alpha')}^{((1^{n_1+n_2},1^{n_2}),\alpha)}$ is of dimension $d_{(\mu,\nu)}^{\alpha} - d_{(\mu',\nu')}^{\alpha'}$.

\subsection{\underline{$\mu=2^{n_2}1^{n_1}, \nu = \emptyset$}} Here $n_2>0$. In this case we have $x\neq 0$, $x^2=0$ and $v\not\in\ker(x)$, so $\ker(x)\cap(\C[x]v)^\perp=\ker(x)\cap(\C v)^\perp$ is a $2n_1+2n_2-1$-dimensional space where the restriction of $\langle,\rangle$ has rank $\max\{2n_1-2,0\}$ ($r=\max\{n_1-1,0\}$). We define $k_1$, $k_2$ and $h$ as in Section \ref{section:xnonzero}. According to Section \ref{subsub:v-not-ker}, there are two possibilities for $(\mu',\nu')$, which we now examine separately.
\subsubsection{}Suppose that $(\mu',\nu')=(1^{n_2'+n_1'},1^{n_2'})=(1^{n_1+n_2-k_1-h},1^{n_2-k_2+h})$
. Then by Section \ref{subsub:v-not-ker}, we have $0\neq xv\in F_k$, and since we know that $xv\in x(F_k^\perp)$ we necessarily have that $xv\in F_k\cap x(F_k^\perp)$, hence we have $k_2-2h>0$ and $n_1'=n_1-k_1+k_2-2h>0$.

We define the variety
\begin{multline*}\mathcal{X}^7(k_2,h)=\{F_k\in\Gr_k^\perp(\ker(x)\cap(\C[x]v)^\perp)~|~\dim(F_k\cap \Im(x))=k_2, \\ \dim(F_k\cap x(F_k^\perp))=k_2-2h,xv\in F_k\cap x(F_k^\perp) \}\end{multline*}
If $xv\in F_k\cap x(F_k^\perp)$, then  $(F_k\cap x(F_k^\perp))/\C xv\in \Gr_{k_{2} - 2h -1}^{\pperp}(\mathbb{C}(xv)^{\pperp}/\mathbb{C}xv).$
It follows that we have a description as an iterated vector bundle
$$\mathcal{X}^7(k_2,h)=(\pi_1^k)^{-1}(\pi_2^{k_2})^{-1}\left(\Gr_{k_2-2h-1}^{\pperp}((\mathbb{C}xv)^{\pperp}/\mathbb{C}xv)\right).$$
The base is an isotropic Grassmannian, hence an irreducible variety of dimension
\begin{equation} 
(k_{2} - 2h-1)(2n_{2} - k_{2}  + 2h -1) - \tfrac12(k_{2}-2h-1)(k_{2}-2h-2).
\end{equation}
Over every point $S\in\Gr_{k_2-2h}^{\pperp}((\mathbb{C}xv)^{\pperp}/\mathbb{C}xv)$, we have $(\pi_2^{k_2})^{-1}(S)\simeq \Gr_{2h}(S^{\pperp}/S)$, so each fibre is irreducible of dimension
$$ 2h(2n_{2} - 2k_{2}+2h).$$
Finally over a point $U\in(\pi_2^{k_2})^{-1}\left(\Gr_{k_2-2h-1}^{\pperp}((\mathbb{C}xv)^{\pperp}/\mathbb{C}xv)\right) $, we have that, when $n_1\geq 1$, then
$(\pi_1^k)^{-1}(U)$ is isomorphic to an open dense subvariety of $$\Gr_{k_1}^\perp(\ker(x)\cap(\C[x]v)^\perp/U)\simeq \Gr_{k_1,2n_1+2n_2-1-k_2}^{\perp,n_1-1}$$
which is irreducible of dimension $k_1(2n_2+2n_1-k_2-k_1-1)-\tfrac{1}{2}k_{1}(k_{1}-1)$. Notice that if $n_1=0$, then necessarily $k_1=0$ and $(\pi_1^k)^{-1}(U)$ is a single point, so the formula for the dimension works in this case as well.
In conclusion, $\mathcal{X}^7(k_2,h)$ is an irreducible variety of dimension 
\begin{align}\label{eq:dim-var-vnotKer-xvinF}
    & (k_{2} - 2h-1)(2n_{2} - k_{2}  + 2h -1) - \tfrac12(k_{2}-2h-1)(k_{2}-2h-2)
+2h(2n_{2} - 2k_{2}+2h)+\\
    &~k_1(2n_2+2n_1-k_2-k_1-1)-\tfrac{1}{2}k_{1}(k_{1}-1). \notag
\end{align}
For $(\mu,\nu)=(2^{n_2}1^{n_1},\emptyset)$, $(\mu',\nu')=(1^{n_1+n_2-k_1-h},1^{n_2-k_2+h})$ we compute the difference:
\begin{equation}\label{eq:diff-dd'-dimX-6}
	\begin{split}
    d_{(\mu,\nu)}^{\alpha} - d_{(\mu',\nu')}^{\alpha'} -\dim\mathcal{X}^7(k_2,h)&=2N(2^{n_2}1^{n_1})- \frac{1}{2} \sum_{i\geq 1} (\alpha_{i}^{2} - \alpha_{i})-2N(2^{n_{2}-k_{2}+h}1^{n_{1}-k_{1}+k_{2}-2h})\\
    &\quad -(n_2-k_2+h)+ \frac{1}{2} \sum_{i\geq 2} (\alpha_{i}^{2} - \alpha_{i})-\eqref{eq:dim-var-vnotKer-xvinF}\\
    & = (n_{1}+ n_{2})(n_{1} + n_{2} - 1) + n_{2}(n_{2}-1) - \tfrac{1}{2}k(k-1) \\
    &\quad -(n_{1}+ n_{2}-k_{1}-h) (n_{1}+ n_{2}-k_{1}-h - 1) \\
    &\quad - (n_{2}-k_{2}+h)(n_{2}-k_{2}+h-1) \\
    & \quad -(n_2-k_2+h)-\eqref{eq:dim-var-vnotKer-xvinF}\\
    &= h(2n_{1} - 2k_{1}+1)+n_2-k_2+h.
    \end{split}
\end{equation}
Since $h\geq 0$, $n_1\geq k_1$, and $n_2-k_2+h\geq 0$ this difference is always greater or equal to zero and it equals zero if and only if $h=0$ and $n_2-k_2+h=0$, (which means $n_2=k_2$).

For fixed $0< n_1'\leq n_1+n_2$, $0\leq n_2'\leq n_2$ such that $n_1'+2n_2'+k=n_1+2n_2$, we have
\begin{align*}\cB_{((1^{n_{1}'+n_2'},1^{n_{2}'}),\alpha')}^{((2^{n_2}1^{n_1},\emptyset),\alpha)}=\coprod_{\substack{\max\{0,k-n_1\}\leq k_2\leq\min\{k,2n_2\} \\ \max\{0,k_2-n_2\}\leq h< \frac{k_2}{2} \\ k_2-h=n_2-n_2'}}\mathcal{X}^7(k_2,h)=\coprod_{\substack{\max\{0,n_2-n_1'-n_2'\}\leq h \\
h\leq\min\{k-(n_2-n_2'), n_2-n_2'-1\}}}\mathcal{X}^7(h+n_2-n_2',h)
\end{align*}
As in previous cases, if $\cX^7(k_2,h)\neq\emptyset\neq\cX^7(k_2+1,h+1)$, then $\cX^7(k_2+1,h+1)\subset\overline{\cX^7(k_2,h)}$. It follows that there is one irreducible piece of maximal dimension that contains all the other ones in its closure as positive codimension subvarieties hence $\cB_{((1^{n_{1}'+n_2'},1^{n_{2}'}),\alpha')}^{((2^{n_2}1^{n_1},\emptyset),\alpha)}$ is irreducible. Notice that by \eqref{eq:diff-dd'-dimX-6} we have that the piece of maximal dimension is $\mathcal{X}(h+n_2-n_2',h)$ for minimal $h$.

If $n_2'=n_2-k_2+h>0$, then all pieces are strictly lower dimensional than $d_{(\mu,\nu)}^{\alpha} - d_{(\mu',\nu')}^{\alpha'}$ and the shape of $(\mu',\nu')$ is not contained in the shape of $(\mu,\nu)$. 

Otherwise, if $n_2'=0$, we have that the minimal $h$ is $\max\{0,n_2-n_1'\}$, and this equals zero if and only if $n_1'\geq n_2$, if and only if the shape $(1^{n_1'},\emptyset)$ is obtained from $(2^{n_2}1^{n_1},\emptyset)$ by removing boxes with no two of them being in the same row. 

In conclusion, $\cB_{((1^{n_{1}'+n_2'},1^{n_{2}'}),\alpha')}^{((2^{n_2}1^{n_1},\emptyset),\alpha)}$ is an irreducible variety of dimension less or equal to $d_{(\mu,\nu)}^{\alpha} - d_{(\mu',\nu')}^{\alpha'}$ with equality if and only if the shape is obtained by removing boxes in the appropriate way.
\begin{example}Here $n_1=2$, $n_2=3$. In the first example $k=4$, $n'_1=2$, $n_2'=1>0$, so the diagrams are not nested. In the second example $k=6$, and, while $n'_2=0$, $n'_1=2<3=n_2$, so some boxes are removed from the same row, while in the third example $k=4$, $n'_2=0$ and $n'_1=4\geq 3=n_2$, so the boxes are removed from different rows of the two tableaux. 
$$(1^{n'_1+n'_2},1^{n'_2})=\left(\young(~,~,~),\young(~)\right)\not\subset (2^{n_2}1^{n_1},\emptyset)=\left(\young(~~,~~,~~,:~,:~),\emptyset\right);$$ 
$$(1^{n'_1+n'_2},1^{n'_2})=\left(\young(~,~),\emptyset\right)\subset (2^{n_2}1^{n_1},\emptyset)=\left(\young(~~,~~,~~,:~,:~),\emptyset\right),\quad \left(\young(\xx ~,\xx ~,\xx \xx,:\xx,:\xx),\emptyset\right);$$
$$(1^{n'_1+n'_2},1^{n'_2})=\left(\young(~,~,~,~),\emptyset\right)\subset (2^{n_2}1^{n_1},\emptyset)=\left(\young(~~,~~,~~,:~,:~),\emptyset\right),\quad \left(\young(\xx ~,\xx ~,\xx ~,:~,:\xx),\emptyset\right).$$
\end{example}
\subsubsection{}Suppose that $(\mu',\nu')=(2^{n_2'}1^{n_1'},\emptyset)=(2^{n_2-k_2+h}1^{n_1-k_1+k_2-2h},\emptyset)$, with $n_2'=n_2-k_2+h>0$. Then by Section \ref{subsub:v-not-ker}, we have $0\neq xv\not\in F_k$. 
We define the variety
\begin{multline*}\mathcal{X}^8(k_2,h)=\{F_k\in\Gr_k^\perp(\ker(x)\cap(\C[x]v)^\perp)~|~\dim(F_k\cap \Im(x))=k_2, \\ \dim(F_k\cap x(F_k^\perp))=k_2-2h,xv\not\in F_k\}.\end{multline*}
Notice that, since $xv\in x(F_k^\perp)=(F_k\cap\Im(x))^{\pperp}$, we have $F_k\cap\Im(x)\subset (\C xv)^{\pperp}$. Also the bilinear form $\llangle~,~\rrangle$ restricted to the $(2n_2 - 1)$-dimensional space $(\C xv)^{\pperp}\subset \Im(x)$ has rank $2n_2-2$ ($r=n_2-1$). It follows that $F_k\cap x(F_k^{\perp})\in Z:= \{S\in\Gr^{\pperp}_{k_2-2h}((\C xv)^{\pperp})~|~S\cap \C xv=0\}$ which is an open dense subvariety of  $\Gr_{k_2-2h}^{\pperp}((\C xv)^{\pperp})$, hence irreducible of dimension
\begin{equation*}
\dim\Gr_{k_2-2h,2n_2-1}^{\perp n_2-1}=(k_{2}-2h)(2n_{2}-1-k_2+2h)-\tfrac{1}{2}(k_{2}-2h)(k_{2}-2h-1).
\end{equation*}
Over every point $S\in Z$, we consider only the points $U\in(\pi_2^{k_2})^{-1}(S)$ such that $xv\not\in U$, and $U\subset (\C xv)^{\pperp}$ i.e.
we are choosing $U/S\in \Gr_{2h}((S^{\pperp}\cap(\C xv)^{\pperp})/S)$, such that $U/S\cap\C xv/S=0$ (where $\C xv/S$ is the image of $\C xv$ under the projection $S^{\pperp}\cap(\C xv)^{\pperp}\to (S^{\pperp}\cap(\C xv)^{\pperp})/S$) so each fibre, if nonempty, is an open subvariety of a Grassmannian, hence irreducible of dimension
$$\dim\Gr_{2h,2n_2-2(k_2-2h)-1}= 2h(2n_{2} - 2k_{2}+2h-1).$$
This fibre is nonempty if and only if it is possible to choose $U/S$ that does not contain $\C xv/S$, if and only if 
\begin{align*}
 2h&<2n_2-2(k_2-2h)-1\\
 2h&<2n_2-2k_2+4h-1\\
 1&<2n_2-2k_2+2h\\
 1/2&<n_2-k_2+h.
\end{align*}
which is equivalent to $n_2-k_2+h>0$ since those are all integers.

Finally, over a point $U\in(\pi_2^{k_2})^{-1}\left(Z\right)$, (with $U\subset (\C xv)^{\pperp}$ and $xv\not\in U$) we have that, when $n_1\geq 1$,
$(\pi_1^k)^{-1}(U)$ is isomorphic to an open dense subvariety of $$\Gr_{k_1}^\perp(\ker(x)\cap(\C[x]v)^\perp/U)\simeq \Gr_{k_1,2n_1+2n_2-1-k_2}^{\perp,n_1-1}$$
which is irreducible of dimension $k_1(2n_2+2n_1-k_2-k_1-1)-\tfrac{1}{2}k_{1}(k_{1}-1)$. If $n_1=0$, then necessarily $k_1=0$ and $(\pi_1^k)^{-1}(U)$ is a single point, so the formula for the dimension works in this case as well.
In conclusion, $\mathcal{X}^8(k_2,h)$ is an irreducible variety of dimension 
\begin{align}\label{eq:dim-var-vnotKer-xvnotF}
    & (k_{2} - 2h)(2n_{2} - k_{2}  + 2h -1) - \tfrac12(k_{2}-2h)(k_{2}-2h-1)
+2h(2n_{2} - 2k_{2}+2h-1)+\\
    &~k_1(2n_2+2n_1-k_2-k_1-1)-\tfrac{1}{2}k_{1}(k_{1}-1). \notag
\end{align}

For $(\mu,\nu)=(2^{n_2}1^{n_1},\emptyset)$, $(\mu',\nu')=(2^{n_2-k_2+h}1^{n_1-k_1+k_2-2h},\emptyset)$ we compute the difference:
\begin{equation}\label{eq:diff-dd'-dimX-7}
	\begin{split}
    d_{(\mu,\nu)}^{\alpha} - d_{(\mu',\nu')}^{\alpha'} -\dim\mathcal{X}^8(k_2,h)&=2N(2^{n_2}1^{n_1})- \frac{1}{2} \sum_{i\geq 1} (\alpha_{i}^{2} - \alpha_{i})-2N(2^{n_{2}-k_{2}+h}1^{n_{1}-k_{1}+k_{2}-2h})\\
    &\quad + \frac{1}{2} \sum_{i\geq 2} (\alpha_{i}^{2} - \alpha_{i})-\eqref{eq:dim-var-vnotKer-xvnotF}\\
    & = (n_{1}+ n_{2})(n_{1} + n_{2} - 1) + n_{2}(n_{2}-1) - \tfrac{1}{2}k(k-1) \\
    &\quad -(n_{1}+ n_{2}-k_{1}-h) (n_{1}+ n_{2}-k_{1}-h - 1) \\
    &\quad - (n_{2}-k_{2}+h)(n_{2}-k_{2}+h-1) -\eqref{eq:dim-var-vnotKer-xvnotF}\\
    &= h(2n_{1} - 2k_{1}+1).
    \end{split}
\end{equation}
Since $h\geq 0$ and $n_1\geq k_1$ this difference is always greater or equal to zero and it equals zero if and only if $h=0$.

For fixed $0\leq n_1'\leq n_1+n_2$, $0< n_2'\leq n_2$ such that $n_1'+2n_2'+k=n_1+2n_2$, we have
\begin{align*}\cB_{((2^{n_{2}'}1^{n_1'},\emptyset),\alpha')}^{((2^{n_2}1^{n_1},\emptyset),\alpha)}=\coprod_{\substack{\max\{0,k-n_1\}\leq k_2\leq\min\{k,2n_2\} \\ \max\{0,k_2-n_2+1\}\leq h\leq \frac{k_2}{2} \\ k_2-h=n_2-n_2'}}\mathcal{X}^8(k_2,h)=\coprod_{\substack{\max\{0,n_2-n_1'-n_2'\}\leq h \\
h\leq\min\{k-(n_2-n_2'), n_2-n_2'-1\}}}\mathcal{X}^8(h+n_2-n_2',h)
\end{align*}
As in previous cases, if $\cX^8(k_2,h)\neq\emptyset\neq\cX^8(k_2+1,h+1)$, then $\cX^8(k_2+1,h+1)\subset\overline{\cX^8(k_2,h)}$. It follows that there is one irreducible piece of maximal dimension that contains all the other ones in its closure as positive codimension subvarieties hence $\cB_{((2^{n_2'}1^{n_{1}',\emptyset}),\alpha')}^{((2^{n_2}1^{n_1},\emptyset),\alpha)}$ is irreducible.
By \eqref{eq:diff-dd'-dimX-7} we have that the piece of maximal dimension is $\mathcal{X}^8(h+n_2-n_2',h)$ for minimal $h$ i.e. $h=\max\{0,n_2-n_1'-n_2'\}$. 
The minimal $h$ is $h=0$ if and only if $n_2\leq n_1'+n_2'$ if and only if the shape $(2^{n_2'}1^{n_1'},\emptyset)$ is obtained from $(2^{n_2}1^{n_1},\emptyset)$ by removing boxes in a vertical strip. In conclusion, 
$\dim\cB_{((\mu',\nu'),\alpha')}^{((\mu,\nu),\alpha)}=d_{(\mu,\nu)}^{\alpha} - d_{(\mu',\nu')}^{\alpha'}$ exactly when the shape is obtained by removing boxes from different rows and is strictly lower otherwise.
\begin{example}Here $n_1=2$, $n_2=3$, $k=4$. In the first example $n'_1+n_2'=0+2<3=n_2$ so some boxes are removed from the same row, while in the second example $n'_1+n'_2=2+1=3\geq n_2$ so the boxes are removed from different rows. 
$$(2^{n'_2}1^{n'_1},\emptyset)=\left(\young(~~,~~),\emptyset\right)\subset (2^{n_2}1^{n_1},\emptyset)=\left(\young(~~,~~,~~,:~,:~),\emptyset\right),\quad \left(\young(~~,~~,\xx\xx,:\xx,:\xx),\emptyset\right);$$
$$(2^{n'_2}1^{n'_1},\emptyset)=\left(\young(~~,:~,:~),\emptyset\right)\subset (2^{n_2}1^{n_1},\emptyset)=\left(\young(~~,~~,~~,:~,:~),\emptyset\right),\quad \left(\young(~~,\xx ~,\xx ~,:\xx,:\xx),\emptyset\right).$$
\end{example}

\section{Extending to any $x^\ell=0$} \label{section:extend}
In this section we outline how our computations for the proof of the case of $x^2=0$ could be extended to any nilpotent $x$. We are not able to carry it out in this paper, but it should be the basis for future work.

Suppose that $x^\ell=0$, $x^{\ell-1}\neq 0$ for some $\ell>0$, then $\lambda=\left(\ell^{n_\ell}\cdots 2^{n_2}1^{n_1}\right)$ for some $n_j\geq 0$ and $J(x)=\left(\ell^{2n_\ell}\cdots 2^{2n_2}1^{2n_1}\right)$. For any isotropic subspace $0\subset F_k\subset F_k^\perp\subset V$ with $F_k\subset \ker(x)$, we define $k_j:=\dim (F_k\cap \Im(x^{j-1}))-\dim (F_k\cap \Im(x^{j}))$ for $j=1,\ldots,\ell$.
Then by Lemma \ref{lem:jtype-quot} we have that
$$J\left(x|_{V/F_k}\right)=J\left(x|_{F_k}\right)=\left(\ell^{2n_\ell-k_\ell}(\ell-1)^{2n_{\ell-1}-k_{\ell-1}+k_\ell}\cdots 2^{2n_2-k_2+k_3}1^{2n_1-k_1+k_2}\right).$$

Now, notice that since $\Im (x)\subset F_k^\perp$, for all $j$ we have that $\Im(x^{j})\subset x^{j-1}(F_k)$. We have then a filtration of subspaces
$$\cdots\subset F_k\cap x^j(F_k^\perp)\subset F_k\cap \Im(x^j)\subset F_k\cap x^{j-1}(F_k^\perp)\subset F_k\cap \Im(x^{j-1})\subset \cdots $$
By Lemma \ref{lem:F-rad-j}, we have $x^j(F_k^\perp)=\left(F_k\cap \Im(x^j)\right)^{\pperp_j}$, so $\left(F_k\cap \Im(x^j)\right)/\left(F_k\cap x^j(F_k^\perp)\right)$ is a symplectic space with non-degenerate skew-symmetric bilinear form given by the restriction of $\llangle~,~\rrangle_j$. We denote the dimension $\dim\left(F_k\cap \Im(x^j)\right)/\left(F_k\cap x^j(F_k^\perp)\right)=2h_j$ (so the $h$ we used in the previous sections would actually become $h_1$).

Then, we apply Lemma \ref{lem:FcapFperp} and, since 
$$a'_{j-1}-a'_j=\begin{cases}k_\ell-2h_{\ell-1} & \text{ if }j=\ell \\ k_{j}+2h_j-2h_{j-1} & \text{ if }1<j<\ell \\ k_1+2h_1 & \text{ if j=1} \end{cases},$$ 
we obtain that 
\begin{align*}&J\left(x|_{F_k^\perp/F_k}\right)=\\
&=(\ell^{2n_\ell-2k_\ell+2h_{\ell-1}}(\ell-1)^{2n_{\ell-1}-2k_{\ell-1}+2k_\ell-4h_{\ell-1}+2h_{\ell-2}}(\ell-2)^{2n_{\ell-2}-2k_{\ell-2}+2k_{\ell-1}+2h_{\ell-3}-4h_{\ell-2}+2h_{\ell-1}}\cdots \\
& \cdots 2^{2n_2-2k_2+2k_3+2h_1-4h_2+2h_3}1^{2n_1-2k_1+2k_2-4h_1+2h_2}) \end{align*}

The difficulty then lies in computing $J\left(x|_{F_k^\perp/F_k+\C[x]v}\right)$ because there are many different possibilities in general for how $v$ can fit within the various subspaces involved in the computation and we do not know how to treat it in general. The case-by-case approach we used for $\ell=1,2$ quickly becomes unwieldy for $\ell>2$.

\section{Exotic Robinson-Schensted-Knuth correspondence} \label{section:ersk}

The Robinson-Schensted correspondence is a combinatorial bijection between permutations and pairs of standard Young tableaux of the same shape, and it was given a geometric interpretation by Steinberg ( see \cite{rs} ) in terms of varieties of pairs of flags in Type A. In \cite{Ros12}, one of the authors generalized the construction to pairs of partial flags, obtaining the Robinson-Schensted-Knuth (RSK) correspondence between matrices with nonnegative integer entries and pairs of semistandard Young tableaux of the same shape. We now outline how Conjecture \ref{conj:main} similarly implies an exotic Robinson-Schensted-Knuth correspondence that generalizes the one obtained for complete flags in \cite{NRS2}.

\begin{definition}[Partial exotic Steinberg Variety]Let $\alpha,\beta$ be compositions of $n$, we define the
 \emph{exotic Steinberg variety of type $(\alpha,\beta)$} to be
$$ \fZ^{\alpha,\beta}:=\{(F_\bullet,F'_\bullet,(v,x))\in \cF^\alpha(V)\times\cF^\beta(V)\times \fN~|~F_\bullet\in\mathcal{C}^\alpha_{(v,x)},~ F'_\bullet\in\mathcal{C}^\beta_{(v,x)}\}.$$
\end{definition}
Notice that for $\alpha=\beta=1^n$, we obtain the exotic Steinberg variety of \cite[Def. 6.1]{NRS1}.

\begin{definition}[Relative Position]Let $\alpha=(\alpha_1,\ldots,\alpha_m)\vDash n$, $\beta=(\beta_1,\ldots,\beta_\ell)\vDash n$, $F\in \cF^\alpha(V)$, $F'\in\cF^\beta(V)$. We define a matrix $A=A(F,F')=(a_{i,j})_{1\leq i\leq 2m,1\leq j\leq 2\ell}$ by
$$a_{i,j}=\dim( (F_{\check{\alpha}_i}\cap F'_{\check{\beta}_j})/ (F_{\check{\alpha}_{i-1}}\cap F'_{\check{\beta}_j}+F_{\check{\alpha}_i}\cap F'_{\check{\beta}_{j-1}}))$$
which we call the \emph{relative position} of the partial flags $F$ and $F'$.    
\end{definition}
\begin{remark}Notice that the row sums of $A(F,F')$ are $\hat{\alpha}=(\alpha_1,\alpha_2,\ldots,\alpha_m,\alpha_m,\ldots,\alpha_1)$ and the column sums are $\hat{\beta}=(\beta_1,\beta_2,\ldots,\beta_\ell,\beta_\ell,\ldots,\beta_1)$. Also, since $F^\perp_{\check{\alpha}_i}=F_{\check{\alpha}_{2m-i}}$ and $F'^\perp_{\check{\beta}_j}=F'_{\check{\beta}_{2\ell-j}}$, it follows that 
\begin{equation}\label{eq:mat-symmetry}a_{2m+1-i,2\ell+1-j}=a_{i,j}.\end{equation}
\end{remark}
For $\alpha,\beta\vDash n$, we define $M^{\alpha,\beta}$ to be the set of matrices with nonnegative integer entries, with row sums $\hat{\alpha}$, column sums $\hat{\beta}$ and satisfying \eqref{eq:mat-symmetry}.
\begin{remark}
The set $M^{\alpha,\beta}$ parametrizes the orbits of the diagonal action of $\Sp_{2n}$ on $\cF^\alpha(V)\times \cF^\beta(V)$. Also, we can identify $M^{\alpha,\beta}$ with the double cosets $W_\alpha \backslash W/ W_\beta$, where $W$ is the Weyl group of Type C (signed permutations) and $W_\alpha$, $W_\beta$ are parabolic subgroups.
\end{remark}

Let $\theta:\fZ^{\alpha,\beta}\to\cF^\alpha(V)\times\cF^\beta(V)$ defined by $\theta(F_\bullet,F'_\bullet,(v,x))=(F,F')$, then for $A\in M^{\alpha,\beta}$ we define $\fZ^{\alpha,\beta}_A=\theta^{-1}(A)$. 

On the other hand, for any bipartition of $n$, $(\mu,\nu)$, and composition $\alpha\vDash n$, we let $\SYB^\alpha(\mu,\nu)$ be the set of all semistandard bitableaux of shape $(\mu,\nu)$ and content $\alpha$. Then for any $T\in \SYB^\alpha(\mu,\nu)$, $T'\in\SYB^{\beta}(\mu,\nu)$, we can consider the subvariety 
$$\fZ^{\alpha,\beta}_{T,T'}:=\{(F,F',(v,x))\in\fZ^{\alpha,\beta}~|~\Phi(F)=T,~\Phi(F')=T'\}.$$

\begin{conj}
The variety $\fZ^{\alpha,\beta}$ (if nonempty) is pure dimensional of dimension
$2n^2-\frac{1}{2}\sum_{i=1}^m(\alpha_i^2-\alpha_i)-\frac{1}{2}\sum_{i=1}^\ell(\beta_i^2-\beta_i)$ and its irreducible components are $\{\overline{\fZ^{\alpha,\beta}_A}~|~A\in M^{\alpha,\beta}\}$. The irreducible components of $\fZ^{\alpha,\beta}$ can also be parametrized as 
$$\{\overline{\fZ^{\alpha,\beta}_{T,T'}}~|~(\mu,\nu)\in\cQ_n,~T\in \SYB^\alpha(\mu,\nu),~T'\in\SYB^{\beta}(\mu,\nu)\}.$$     
\end{conj}
This would follow from Conjecture \ref{conj:main} and give a geometrically defined bijection
$$M^{\alpha,\beta}\longleftrightarrow \coprod_{(\mu,\nu)\in\cQ_n}\SYB^\alpha(\mu,\nu)\times\SYB^{\beta}(\mu,\nu)$$
given by $A\leftrightarrow (T,T')$ if and only if $\overline{\fZ^{\alpha,\beta}_A}=\overline{\fZ^{\alpha,\beta}_{T,T'}}$. We expect that its combinatorial description should be a `semistandardization' of the algorithm of \cite{NRS2}, analogously to what was done in \cite{Ros12} for the RSK correspondence in Type A.
\begin{example}
Let $n=2$, and $\alpha=\beta=(2)$, then there are three pairs of semistandard bitableaux of the same shape with contents $\alpha$ and $\beta$, which are 
$$ \left(\left(\young(1),\young(1)\right),\left(\young(1),\young(1)\right)\right),\quad\left(\left(\young(1,1),\varnothing\right),\left(\young(1,1),\varnothing\right)\right),\quad \left(\left(\varnothing,\young(1,1)\right),\left(\varnothing,\young(1,1)\right)\right).$$
It can then be checked that the geometric correspondence matches these respectively to the following matrices of relative positions (which comprise the set $M^{(2),(2)}$):
$$ \begin{bmatrix} 2 & 0 \\ 0 & 2 \end{bmatrix} \qquad \begin{bmatrix} 1 & 1 \\ 1 & 1 \end{bmatrix} \qquad \begin{bmatrix} 0 & 2 \\ 2 & 0 \end{bmatrix}.$$
\end{example}

\bibliographystyle{alpha}

\end{document}